\documentclass[a4paper,11pt]{article}
\usepackage[utf8]{inputenc}
 \usepackage[T1]{fontenc}
 \usepackage[normalem]{ulem}
 \usepackage[english]{babel}
\usepackage{hyperref,float}
\usepackage{graphicx}
\usepackage{amsmath}
\usepackage{amssymb}
\usepackage{amsthm,mathtools}
\usepackage{graphicx}
\usepackage{pstricks}
\usepackage{amsfonts}
\usepackage{xr}
\usepackage{color}
\usepackage{graphicx,epsfig}
\usepackage{pdfsync,subfigure}
\usepackage{dsfont}
\usepackage{array}

\newcommand{\R}{\mathbb{R}}

\newcommand{\e}{\varepsilon}
\renewcommand{\epsilon}{\varepsilon}
\newcommand{\eps}{\varepsilon}

\newcolumntype{M}[1]{>{\raggedright}m{#1}}

\newtheorem{lem}{Lemma}
\newtheorem{prop}{Proposition}

\newtheorem{theo}{Theorem}
\theoremstyle{definition}

\newtheorem{rmk}{Remark}

\newtheorem{ex}{Example}

\newcommand{\fer}[1]{(\ref{#1})}
\newcommand{\commentout}[1]{}

\newcommand {\al} {\alpha}

\newcommand {\vp} {\varphi}

\newcommand {\Chi} {{\bf \raise 2pt \hbox{$\chi$}} }
\newcommand {\f}   {\frac}
\newcommand {\p}   {\partial}

\newcommand{\beq}{\begin{equation}}
\newcommand{\eeq}{\end{equation}}
\newcommand{\beqa}{\begin{eqnarray}}
\newcommand{\bea} {\left\{ \begin{array}{l}}
\newcommand{\beqan}{\begin{eqnarray*}}
\newcommand{\eeqa}{\end{eqnarray}}
\newcommand{\eeqan}{\end{eqnarray*}}
\newcommand{\eea} {\end{array} \right.}
\title{A Hamilton-Jacobi approach for a model of population structured by space and trait}
\author{Emeric Bouin \thanks{Corresponding author}\,
\thanks{UMR CNRS 5669 'UMPA' and INRIA project 'NUMED', \'Ecole Normale Sup\'erieure de Lyon, 
46, all\'ee d'Italie, 
F-69364 Lyon Cedex 07, 
France. Email : emeric.bouin@ens-lyon.fr}
\and
Sepideh Mirrahimi \thanks{
CNRS, Institut de Math\'ematiques de Toulouse UMR 5219, 31062 Toulouse, France. Email: sepideh.mirrahimi@math.univ-toulouse.fr}
\thanks{ Universit\'e de Toulouse ; UPS, INSA,  UT1, UTM ; IMT ; 31062 Toulouse, France.}
}

\date{\today}
\begin{document}

\maketitle
\begin{abstract}
We study a non-local parabolic Lotka-Volterra type equation describing a population structured by a space variable $x \in \R^d$ and a phenotypical trait  $\theta \in \Theta$. Considering diffusion, mutations and space-local competition between the individuals, we analyze the asymptotic (long--time/long--range in the $x$ variable) exponential behavior of the solutions. Using some kind of real phase WKB ansatz, we prove that the propagation of the population in space can be described by a Hamilton-Jacobi equation with obstacle which is independent of $\theta$. The effective Hamiltonian is derived from  an eigenvalue problem.
 \\
The main difficulties are  the lack of regularity estimates in the space variable, and the lack of comparison principle due to the non-local term.
\end{abstract}
\noindent{\bf AMS Class. No:} {45K05, 35B25, 49L25, 92D15, 35F21.}
\pagestyle{plain}
\pagenumbering{arabic}
\section{Introduction}

It is known that the asymptotic (long-time/long-range) behavior of the solutions of some reaction-diffusion equations, as KPP type equations, can be described by level sets of  solutions of some relevant Hamilton-Jacobi equations (see \cite{WF.PS:86,LE.PS:89,GB.LE.PS:90,GB.PS:94,GB.PS:97,PS:98}). These equations, which admit traveling fronts as solutions, can be used as models in ecology to describe dynamics of a population structured by a space variable.\\
 A related, but different, method using Hamilton-Jacobi equations with constraint has been developed recently to study populations structured by a phenotypical trait (see \cite{OD.PJ.SM.BP:05,GB.BP:07,GB.BP:08,GB.SM.BP:09,AL.SM.BP:10}). This approach provides an asymptotic study of the solutions in the limit of small mutations and in long time, and shows that the asymptotic solutions concentrate on one or several Dirac masses which evolve in time.\\
 Is it possible to combine these two approaches to study populations structured at the same time by a phenotypical trait and a space variable?\\

A challenge in evolutionary ecology is to provide and to analyze models that take into account jointly the evolution and the spatial structure of a population. Most of the existing models whether concentrate on the evolution and neglect or simplify significantly the spatial structure, or deal only with the spatial dynamics of a population neglecting the impact of evolution on the dynamics. However, to describe many phenomena in ecology, as the spatial structure and the local adaptation of species  \cite{JE.DD.TC.YA.TO:07},  to understand the effect of environmental changes on a population \cite{AD.FM.IC.MK.OR:12} or to estimate the propagation speed of an invasive species \cite{BP.GB.JW.RS:06,EB.VC.NM.SM:12},  it is crucial to consider the interactions between  ecology and evolution.  We refer also to \cite{SL.MB:08} and the reference therein for general literature on the subject. 
\\

In this paper, we study a  population that is structured by a continuous phenotypical trait $\theta \in \Theta$, where $\Theta$ is a smooth and convex bounded subset of $\R^n$, and a space variable $x \in \R^d$. The individuals having a trait $\theta$ at time $t$ and position $x$ are denoted by $n(t,x,\theta)$. We assume that the population moves (in space) with a diffusion process of diffusivity $D > 0$, and that they are subject to mutations, which are also described by a diffusion term with diffusivity $\alpha > 0$. We assume that the individuals in the same position are in competition with all {other} individuals, independently of their trait, and with a constant competition rate $r$. Let us notice that the non-locality in the model comes from here.
We denote by  $r a(x,\theta) \in \mathcal{C}^2\left( \R^d \times \Theta \right)$, the growth rate of trait $\theta$ at position $x$, allowing, in this way, heterogeneity in space. The model reads

\beq\label{main1}
\begin{cases}
\p_t n = D \Delta_{xx} n + \al \Delta_{\theta\theta} n + r n (a(x,\theta)-\rho ), \qquad { (t,x,\theta)\in (0,\infty)\times \R^d \times \Theta} \medskip,\\
\displaystyle\frac{\p n}{\p \textrm{\bf{n}}} = 0  \qquad \textrm{ on } {(0,\infty)\times \R^d \times\partial \Theta},\medskip\\
n(0,x,\theta)=n^0(x,\theta),  \qquad (x,\theta)\in \R^d \times \Theta.
\end{cases}
\eeq

We assume Neumann boundary conditions in the trait variable, meaning that the available traits are given by the set $\Theta$. Moreover, the initial condition $n^0$ is given and nonnegative. The variable $\rho$ stands for the total density:
\begin{equation*}
\forall (t,x) \in \R^+ \times \R^d,\qquad  \rho(t,x)=\int_{\Theta} n(t,x,\theta)d\theta.
\end{equation*}

This is not the only way to couple the spatial and trait structures. One could also consider a dependence in $\theta$ or $x$ in the spatial diffusivity coefficient, the mutation rate or the competition rate. See for instance \cite{EB.VC.NM.SM:12} for a formal study of a model where the spatial diffusivity rate depends on the trait but the growth rate is homogeneous in space. Although, there have been some attempts to study models structured by trait and space (see for instance \cite{NC.SM:07,AA.LD.CP:12,OB.VC.NM.RV:12,EB.VC.NM.SM:12,HB.GC:12}), not many rigorous studies seem to have analyzed the dynamics of a population continuously structured  by trait and by space, with non-local interactions. However, a very recent article \cite{MA.JC.GR:13}, which studies a model close to \fer{main1} with some particular growth rate $a(x,\theta)$, proves existence of traveling wave solutions. Here, we consider a different approach where we perform an asymptotic analysis. Our objective is to generalize the methods developed recently on models structured only by a phenotypical trait \cite{GB.BP:08,GB.SM.BP:09,AL.SM.BP:10} to spatial models, to be able to use the previous results  in more general frameworks. Moreover, this approach allows us to study models with general growth rates $a$, where the speed of propagation is not necessarily constant. See also   \cite{SM:12} for another work in this direction, where the Hamilton-Jacobi approach is used to study a population model with a  discrete spatial structure. 
\\

We expect that the population described by \fer{main1} propagates in the $x$-direction and that it attains a certain distribution in $\theta$ in  the invaded parts. We seek for such behavior by performing an asymptotic analysis of the following rescaled model which corresponds to considering small diffusion in space and long time:
\beq\label{main2}
\begin{cases}
\e\p_t n_\e=\e^2 D \Delta_{xx} n_\e+\al \Delta_{\theta\theta} n_\e+r n_\e(a(x,\theta)-\rho_\e), \qquad {(t,x,\theta)\in (0,\infty)\times\R^d \times \Theta} \medskip,\\
\displaystyle\frac{\p n_\e}{\p \textrm{\bf{n}}} = 0   \qquad\text{ on } {(0,\infty)\times \R^d \times\partial \Theta} \medskip,\\
n_\e(0,x,\theta)=n_\e^0(x,\theta),  \qquad  (x,\theta)\in \R^d \times \Theta.
\end{cases}
\eeq
The purpose of this work is to derive rigorously the limit $\eps \to 0$ in \eqref{main2}. Our study is based on the usual \textit{Hopf-Cole transformation} which is used in several works on reaction-diffusion equations (as for front propagation in \cite{WF.PS:86,LE.PS:89,GB.LE.PS:90}), in the study of parabolic integro-differential equations modeling populations structured by a phenotypical trait (see e.g. \cite{OD.PJ.SM.BP:05,GB.BP:08}) and also recently in the study of the hyperbolic limit of some kinetic equations \cite{EB.VC:12}:
\begin{equation}\label{WKB}
u_{\eps} := \eps \ln n_\e, \qquad \textrm{or equivalently,} \qquad n_\e=\exp\left(\f{u_\e}{\e}\right).
\end{equation}
Thanks to standard maximum principle arguments, $n_{\eps}$ is nonnegative. The quantity $u_{\eps}$ is then well defined for all $\eps > 0$. By replacing \eqref{WKB} in \eqref{main2} we obtain
\beq\label{ue}
\begin{cases}
\p_t u_\e=\e D\Delta_{xx} u_\e+\f{\al}{\e} \Delta_{\theta\theta} u_\e+D|\nabla_x u_\e|^2+\f{\al}{\e^2}|\nabla_\theta u_\e|^2+r (a(x,\theta)-\rho_\e),  \qquad{ \left( t,x,\theta \right) \in (0,\infty)\times \R^d \times \Theta}, \medskip\\
\displaystyle\frac{\p u_\e}{\p \textrm{\bf{n}}} = 0  \qquad\text{ on }{(0,\infty)\times \R^d \times\partial \Theta},\medskip\\
u_\e(0,x,\theta)= u_\e^0(x,\theta)  \qquad  \left(x,\theta\right) \in \R^d \times \Theta.
\end{cases}
\eeq
Throughout the paper, we will use the following assumptions:
\begin{equation}\label{as:u0}
\forall \eps > 0, \quad \forall x \in \R^d, \qquad - C_1(x) \leq u_{\eps}^0 \leq C. 
\end{equation}
\beq\label{as:ini2}
\lim_{\e\to 0} u_\e^0(x,\theta)=u_0(x),\qquad \text{uniformly in $\theta\in \Theta$}.
\eeq
\beq\label{as:amax} \forall (x , \theta ) \in \R^d \times \Theta, \qquad \psi(x)=-M |x|^2 + B \leq a(x,\theta) - a_{\infty }< 0,\eeq
for some $a_{\infty} \in \R$. We also suppose the two following bounds:
\beq\label{as:ad} \Vert \nabla_\theta a(\cdot,\cdot) \Vert_{\infty} = b_{\infty}.\eeq
\beq\label{as:ini}\forall x \in \R^d,  \qquad  \rho_\e^0(x) \leq a_{\infty}.\eeq

To state our results we first need the following lemma:

\begin{lem}{{\bf (Eigenvalue problem).}}\label{eigenHJ}

For all $x \in \R^d$, there exists a unique eigenvalue $H(x)$ {corresponding to a strictly} positive eigenfunction $Q(x,\cdot) \in \mathcal{C}^0(\Theta)$ which satisfies
\begin{equation}
\label{ev}
 \begin{cases}
\displaystyle\alpha \Delta^2_{\theta\theta} Q + r a(x,\cdot) Q = H(x) Q,  \qquad  {\text{in }\Theta},\medskip\\
\displaystyle\frac{\p Q(x,\cdot)}{\p \textbf{n}} = 0  \qquad{\text{on } \partial \Theta}.
\end{cases}
\end{equation}
The eigenfunction is unique under the additional normalization assumption
\begin{equation} 
\label{normQ}
\forall x \in \R^d, \qquad \int_{\Theta} Q(x,\theta) d \theta = 1.
\end{equation}
 Moreover, $H$ and $Q$ are smooth functions.
\end{lem}

We note that in this article, we suppose that $\Theta$ is bounded to avoid technical difficulties. However, we expect that the results would remain true for unbounded domains $\Theta$ under suitable coercivity conditions on $-a$ such that the spectral problem \fer{ev} has a unique solution.\\

We can now state our main result: 

\begin{theo}{\bf (Asymptotic behavior).}\label{th:main}
 Assume \fer{as:u0}--\fer{as:ini}. Then
\begin{itemize}
\item[(i)] The family $(u_\e)_\e$ converges locally uniformly to $u:[0,\infty)\times\R^d \to \R$ the unique viscosity solution of 
\begin{equation}\label{varHJ2}
\begin{cases}\max (\p_t u- D|\nabla_x u|^2-H,u)=0,  \qquad \text{in $(0,\infty)\times\R^d$},\medskip\\
{u(0,\cdot)=u_0(\cdot)} \qquad  {\text{in $\R^d$}}.
\end{cases}
\end{equation}
\item[(ii)]  Uniformly on compact subsets of $\text{Int}\left\lbrace u < 0\right\rbrace \times \Theta$, \quad $\lim_{\eps \to 0} n^{\eps} = 0$,
\item[(iii)] For every  compact subset $K$ of $\mathrm{Int} \left( \left\lbrace u(t,x) = 0\right\rbrace \cap \left\lbrace H(x) > 0 \right\rbrace \right)$, there exists ${\overline C} > 1$ such that, 
\beq 
\label{rhoinf}
\liminf_{\eps \to 0} \rho_\eps (t,x)\geq \frac{H(x)}{r {\overline C}},\qquad \text{uniformly in $K$}.
\eeq

\end{itemize}
\end{theo}

We notice that $u$ does not depend on $\theta$ and therefore, we do not have any supplementary constraint in \fer{varHJ2} due to the boundary. 
The variational equality \fer{varHJ2} gives indeed the effective propagation behavior of the population; the zero level-sets of $u$ indicate where the population density is asymptotically positive (see also Lemma \ref{grate}). We recall that the effective Hamiltonian $H$ in \fer{varHJ2} is defined by the spectral problem \fer{ev} which hides the information on  the trait variability.
 \\


To understand Theorem \ref{th:main}, it is illuminating to provide the following heuristic argument. We write a \textit{formal} expansion of $u_\e$:
\begin{equation*}
u_\e(t,x,\theta) = u_0(t,x,\theta) + \e u_1(t,x,\theta) + \mathcal{O}(\e^2).
\end{equation*}
Replacing this in \fer{ue} and keeping the terms of order $\e^{-2}$ we obtain, for all $(t,x,\theta)$,
\begin{equation*}
\left\vert \nabla_\theta u_0(t,x,\theta) \right\vert^2=0.
\end{equation*}
This suggests that $u_0$ should be independent of $\theta$: $u_0(t,x,\theta)=u_0(t,x)$. Next, keeping the zero order terms (terms with coefficient $\e^0$), yields:
\begin{equation}\label{heur1}
- \alpha \left( \Delta_\theta u_1 + |\nabla_\theta u_1|^2 \right) - r a(x,\theta)=\left[-\p_t u_0+ D|\nabla_x u_0|^2-r\rho_0\right](t,x).
\end{equation}
Here, $\rho_0$ denotes the \textit{formal} limit of $\rho_\eps$ when $\eps \to 0$. Moreover, $u_1$ satisfies Neumann boundary conditions. Since the r.h.s. of \eqref{heur1} is independent of $\theta$, Lemma \ref{eigenHJ}  implies
\begin{equation*}
\left[\p_t u_0- |\nabla_x u_0|^2+r\rho_0 \right](t,x)=H(x)
\quad
\text{and}
\quad
u_1(t,x,\theta)=\ln Q(x,\theta)+\mu(t,x).
\end{equation*}
We can now write 
\begin{equation*}
n_\e(t,x,\theta)\approx e^{\f{u_0(t,x)}{\e}+u_1(t,x,\theta)}, \qquad \rho_\eps(t,x) \approx e^{\mu(t,x)+\f{u_0(t,x)}{\e}}.
\end{equation*}
As a consequence, $\rho_\eps$ uniformly bounded implies that $u_0$ is nonpositive. Furthermore 
\begin{equation*}
\rho_\e > 0 \qquad {\Longrightarrow} \qquad u_0= 0.
\end{equation*}
We deduce that
\begin{equation*}
\begin{cases}
\rho_0(t,x)=0  \quad \Longrightarrow \quad  \p_t u_0(t,x)- D|\nabla_x u_0|^2(t,x)-H(x)=0,\medskip\\
\rho_0(t,x)>0 \quad  \Longrightarrow \quad  u_0(t,x)=0\quad \text{and}\quad r\exp(\mu(t,x)) = r \rho_0(t,x) = H(x),
\end{cases}
\end{equation*}
and thus
\begin{equation*}
\max \left( \p_t u_0- D |\nabla_x u_0|^2-H(x)   \,,\, u_0 \right)=0.
\end{equation*}
Moreover the above arguments suggest that
\begin{equation*}
n_\e(t,x,\theta) \underset{\eps \to 0}{\longrightarrow} \begin{cases}\frac{H(x)}{r}Q(x,\theta) & \text{if $u_0(t,x)=0$}, \medskip\\
0&\text{if $u_0(t,x)<0$,}
\end{cases}
\end{equation*}
with $Q$ and $H$ given by Lemma \ref{eigenHJ}. We notice finally that, the roles of the trait variable $\theta$ and the spectral problem \eqref{ev} are  respectively similar to  the ones of the fast variable and the cell problem in homogenization theory.\\

Theorem \ref{th:main} does not provide the limits of $\rho_\e$ and $n_\e$ in $\mathrm{Int} \left( \left\lbrace u(t,x) = 0\right\rbrace \cap \left\lbrace H(x) > 0 \right\rbrace \right)$. The determination of such limits in the general case, as was obtained for instance in \cite{LE.PS:89}, is beyond the scope of the present paper. The difficulty here is the lack of regularity estimates in the $x$-direction and the lack of comparison principle for the non-local equation \fer{main2}.  This difficulty also appears in the study of propagating wave solutions of \fer{main1} (see \cite{MA.JC.GR:13}),  where it is not clear whether the propagating front is monotone and the density and the distribution of the population at the back of the front is unknown. However, in Section \ref{asympt} (see Proposition \ref{convvarsep}), we prove the convergence of $n_\e$ and $\rho_\e$ in a particular case. The numerical results in Section \ref{num} suggest that such limits might hold in general.\\

We emphasize that \eqref{main2} does not admit a comparison principle which leads to technical difficulties. This is not only due to the presence of a non-local term but also due to the structure of the reaction term. We refer to \cite{CC.SH.CS:12,EB.VC.GN:13} for models admitting comparison principle although the reaction terms contain non-localities. \\

To prove the convergence of $(u_\e)_\e$ in Theorem \ref{th:main}, we use some regularity estimates that we state below.

\begin{theo}{\bf (Regularity results for $u_{\e}$).}\label{regularity}
Assume \fer{as:u0}, \fer{as:amax}, \fer{as:ad}, \fer{as:ini}.
Then the family $( u_{\eps} )_{\eps > 0}$ is uniformly locally bounded in $\R^+ \times \R\times \Theta$. More precisely, the following inequalities hold:
\begin{equation}\label{boundue}
{\forall (t,x,\theta) \in \R^+ \times \R^d \times \Theta, \qquad }r \psi(x) t - C_1(x) - r \e D M {t^2} \leq u_{\e}(t,x,\theta) \leq C + r a_{\infty} t,
\end{equation}
where $\psi(x) := -Mx^2 + B$.
Next, let $\gamma > 0$ and for all $\eps > 0$, $v_\e:=\sqrt{ C + r a_{\infty} t + \gamma^2 -u_\e}$. Then, for all $\eps > 0$, the following bound holds: 
\begin{equation}\label{lipv}
\forall (t,x,\theta) \in \R^+ \times \R^d \times \Theta, \qquad \vert \nabla_\theta v_\e \vert \leq \frac{\eps}{2 \sqrt{\alpha t}} + \left( \frac{r b_{\infty} \eps^2}{\alpha \gamma}\right)^{\frac13} 
\end{equation}
In particular, this gives a regularizing effect in trait for all $t>0$, and the fact that $ \nabla_\theta v_\e$ converges locally uniformly to 0 when $\eps$ goes to 0.
\end{theo}

We notice from \fer{lipv} that, the limit of $(v_\e)_\e$ (and consequently the limit of $(u_\e)_\e$) as $\e\to 0$, is independent of $\theta$ for all $t>0$, while we do not make any regularity assumption on the initial data.
To obtain the regularizing effect in $\theta$, we provide a Lipschitz estimate on a well-chosen auxiliary function $v_\e$ instead of $u_\e$,  using the Bernstein method \cite{C.I.L:92}.
%
Note that, we do not have any estimate on the derivative of $u_\e$ with respect to $x$ due to the dependence of $\rho_\e$ on $x$. Therefore, we cannot prove the convergence of the $u_\e$'s  as stated in Theorem \ref{th:main} directly from the regularity estimates above. {For this purpose}, we use the so called  half-relaxed limits method for viscosity solutions, see \cite{GB.BP:88}. Moreover, to prove the convergence to the Hamilton-Jacobi equation \fer{varHJ2} we are inspired from the method of perturbed test functions in homogenization \cite{LE:89}.  \\

Finally, the family $\left( u_{\e} \right)_{\e}$ being locally uniformly bounded from Theorem \ref{regularity}, we can introduce its upper and lower semi-continuous envelopes that we will use through the article:
\begin{equation*} 
\underline u(t,x,\theta) := \underset{\underset{(s,y,\theta')\to (t,x,\theta)}{\e\to 0}}{ \underline \liminf} u_\e(s,y,\theta'),
\end{equation*}
\begin{equation*}
\overline u(t,x,\theta) := \underset{\underset{(s,y,\theta')\to (t,x,\theta)}{\e\to 0}}{ \overline \limsup} u_\e(s,y,\theta').
\end{equation*}

Thanks to Theorem \ref{regularity}, we know that $\vert \nabla_{\theta} u_{\e} \vert \to 0$ as $\eps \to 0$, for all $t>0$. As a conclusion, the previous limits do not depend on the variable $\theta$. We have, for all $\theta \in \Theta${, $x \in \R^d$} and $t>0$,
\begin{equation}\label{uinf} 
\underline u(t,x,\theta) = \underline u(t,x) = \underset{\underset{(s,y)\to (t,x)}{\e\to 0}}{ \underline \liminf} u_\e(s,y,\theta),
\end{equation}
\begin{equation}\label{usup}
\overline u(t,x,\theta) = \overline u(t,x)  = \underset{\underset{(s,y)\to (t,x)}{\e\to 0}}{ \overline \limsup} u_\e(s,y,\theta),
\end{equation}

The remaining part of the article is organized as follows. Section \ref{reg} is devoted to the proof of Lemma \ref{eigenHJ} and Theorem \ref{regularity}. The  convergence to the Hamilton-Jacobi equation  (the first part of Theorem \ref{th:main}) is proved in Section \ref{conv}. In Section \ref{asympt}, using the Hamilton-Jacobi description, we study the limits of $n_\e$ and $\rho_\e$ and in particular complete the proof of Theorem \ref{th:main}. We also provide some qualitative properties on the effective Hamiltonian $H$ and the corresponding eigenfunction $Q$ in Section \ref{Qualitative}. Finally, in Section \ref{Examples} we give some examples and comments on the spectral problem, and some numerical illustrations for the time-dependent problem.

\section{Regularity results ({The p}roof of Theorem \ref{regularity})}\label{reg}

In this {s}ection we prove Theorem \ref{regularity}. To this end, we first provide a uniform upper bound on $\rho_\e$ (see Lemma \ref{rhobound}). Next, using this estimate we give uniform upper and lower bounds on $u_\e$. Finally we prove a Lipschitz estimate with respect to $\theta$ on $u_\e$. \\

\begin{lem}{\bf (Bound on $\rho_{\e}$).}\label{rhobound}

Assume \fer{as:amax} and \fer{as:ini}. Then, for all $\eps > 0$, the following \textit{a priori} bound holds :

\begin{equation}\label{rho}
\forall (t,x) \in \R^+ \times \R^d, \qquad 0 \leq \rho_{\e}(t,x) \leq a_{\infty}.
\end{equation}
\end{lem}

\begin{proof}[\bf Proof of Lemma \ref{rhobound}]
The nonnegativity follows directly from the nonnegativity of $n_{\eps}$. The upper bound can be derived using the maximum principle. We  show indeed that $\rho_{\e}$ is a subsolution of a suitable Fisher-KPP equation. We integrate \eqref{main2} in $\theta$ to obtain

\begin{equation*}\label{KPP}
\e\p_t \rho_\e=\e^2 D \Delta_x\rho_\e+r\left( \int_{\Theta} n_\e(t,x,\theta) a(x,\theta)d\theta-\rho_\e^2\right).
\end{equation*}
Using \eqref{as:amax} and the non negativity of $n_{\eps}$, we deduce
\begin{equation*}\label{subKPP}
\e\p_t \rho_\e \leq \e^2 D \Delta_x^2\rho_\e+r \rho_{\e} \left( a_{\infty} -\rho_{\e}\right),
\end{equation*}
so that the maximum principle and \eqref{as:ini} ensure
\begin{equation*}
\forall (t,x) \in \R^+ \times \R^d, \qquad \rho_\e(t,x)\leq a_{\infty}.
\end{equation*}

\end{proof}

We can now proceed with the proof of Theorem \ref{regularity}. For legibility, we divide the proof into several steps as follows.\\

{\bf \# Step 1.  Upper bound on $u_{\e}$.}
Define $\tilde u_{\e} := u_{\e} - r a_{\infty} t$.
Using \eqref{ue}, we find
\begin{equation*}\label{tildeue}
\p_t \tilde u_\e = \e D\Delta_x \tilde u_\e+\f{\al}{\e} \Delta_\theta \tilde u_\e+D|\nabla_x  \tilde u_\e|^2+\f{\al}{\e^2}|\nabla_\theta  \tilde u_\e|^2+r (a(x,\theta)-a_{\infty}) - r \rho_{\e}.
\end{equation*}
Then, we conclude from \fer{as:amax}, \fer{as:u0} and the maximum principle that
\begin{equation*}
\forall(t,x,\theta) \in \R^+ \times \R^d \times \Theta, \qquad u_{\e}(t,x,\theta) \leq u_{\e}^{0}(x,\theta) + r a_{\infty}t \leq C + r a_{\infty} t.  
\end{equation*}

{\bf \# Step 2. Lower bound on $u_{\e}$.}
From \eqref{ue}, \fer{as:amax} and Lemma \ref{rhobound} we can write
\begin{equation*}
\p_t u_\e \geq \e D\Delta_xu_\e+\f{\al}{\e} \Delta_\theta u_\e + r (a(x,\theta)- a_{\infty}) \geq \e D\Delta_x u_\e+\f{\al}{\e} \Delta_\theta u_\e + r \psi(x),
\end{equation*}
with $\psi(x)=-M |x|^2+B$.
Next we rewrite the above inequality, in terms of  $q_{\e} := u_{\e} - r \psi(x)  t + \e r M D {t^2}$:
\begin{equation*}
\p_t q_\e \geq \e D\Delta_x q_\e + \f{\al}{\e} \Delta_\theta q_\e + r \e \psi''(x) D t + 2 r \e M D t  \geq \e D\Delta_x q_\e + \f{\al}{\e} \Delta_\theta q_\e 
\end{equation*}
Finally the maximum principle combined with Neumann boundary conditions and \fer{as:u0} imply that 
\begin{equation*}
u_{\e} \geq u_{\e}^0 + r \psi(x) t - r \e D M {t^2} \geq - C_1(x) + r \psi(x) t - r \e D M {t^2}.
\end{equation*}

{\bf \# Step 3. Lipschitz bound.}
We conclude the proof of Theorem \ref{regularity} by using the Bernstein method {\cite{C.I.L:92}} to obtain a regularizing effect with respect to the variable $\theta$. The upper bound \fer{boundue} proved above ensures that  the function $v_{\e}$ is well-defined. We then rewrite \fer{ue} in terms of $v_\e$: 
\begin{equation*}
\p_t v_\e=\e D \Delta_x v_\e+\f{\al}{\e} \Delta_\theta v_\e + D \left( \f{\e}{v_\e}-2v_\e \right)|\nabla_x  v_\e|^2+
\left(  \f{\al}{\e v_\e}-\f{2\al v_\e}{\e^2}  \right)|\nabla_\theta  v_\e|^2 - \frac{1}{2 v_{\e}} r (a(x,\theta)- a_{\infty} - \rho_\e).
\end{equation*}
We differentiate the above equation with respect to $\theta$ and 
multiply it by $\f{\nabla_\theta  v_\e}{|\nabla_\theta  v_\e|}$ to obtain
\begin{multline}\label{eq2}
\p_{t} |\nabla_\theta  v_{\e}| \leq \e D \Delta_x | \nabla_\theta  v_\e| + \f{\al}{\e} \Delta_\theta  |\nabla_\theta  v_\e| + 2D\left( \f{\e}{v_\e}-2v_\e \right)\nabla_x  v_\e \cdot \nabla_x  | \nabla_\theta   v_\e | \\+ 2\left(  \f{\al}{\e v_\e}-\f{2\al v_\e}{\e^2}  \right)\nabla_\theta  v_\e \cdot \nabla_\theta  | \nabla_\theta  v_\e | + D \left( -\f{\e}{v_\e^2}-2 \right)|\nabla_x  v_\e|^2 | \nabla_\theta  v_\e | \\+ \left(  -\f{\al}{\e v_\e^2}-\f{2\al }{\e^2}  \right)|\nabla_\theta  v_\e|^3 + \frac{r |\nabla_\theta  a(x,\theta)|}{2 v_{\e}},
\end{multline}
since the last contribution of the r.h.s of the above inequality becomes nonpositive.
From \fer{as:ad} and \fer{boundue}, it follows that $w_{\e} := |\nabla_\theta  v_\e|$ is a subsolution of the following equation
{\begin{multline}\label{love}
\p_{t} w_{\e} \leq \e D \Delta_x w + \f{\al}{\e} \Delta_\theta  w_{\e} + 2D\left( \f{\e}{v_\e}-2v_\e \right)\nabla_x  v_\e \cdot \nabla_x  w_{\e} \\+ 2\left(  \f{\al}{\e v_\e}-\f{2\al v_\e}{\e^2}  \right)\nabla_\theta  v_\e \cdot \nabla_\theta  w_{\e} - \f{2\al }{\e^2}  |w_{\e}|^3 + \frac{r b_{\infty}}{2 \gamma}.
\end{multline}}
The last step is now to prove that $z(t) := \frac{\eps}{2 \sqrt{\alpha t}} + \left( \frac{r b_{\infty} \eps^2}{\alpha \gamma}\right)^{\frac13}$ is a supersolution of \eqref{love}. We compute
{\begin{equation*}
z'(t) + \frac{2 \alpha }{ \eps^2 } (z(t))^3 = \frac{2 \alpha }{ \eps^2 } \left( z(t)^3 - \left( z(t) - \left( \frac{r b_{\infty} \eps^2}{\alpha \gamma}\right)^{\frac13}  \right)^3 \right)
\geq \frac{r b_{\infty}}{2 \gamma}.
\end{equation*}}
The Neumann boundary {condition for $u_\e$ implies} a Dirichlet boundary condition for $w_\e$. Thus, \eqref{lipv} follows from the comparison principle.

\section{Convergence to the Hamilton-Jacobi equation  (The proof of  Theorem \ref{th:main}--(i))}\label{conv}

In this section, we first prove Lemma \ref{eigenHJ}. Next, using the regularity estimates obtained above we prove the convergence of  $(u_\e)_\e$ to the solution of \fer{varHJ2} (Theorem \ref{th:main} (i)).  This will be derived from the following proposition which also provides a partial result, once we relax assumption \fer{as:ini2}:

\begin{prop}{\bf (Convergence to the Hamilton-Jacobi equation).}\label{HJlim}

\begin{itemize}
\item[(i)] Assume \fer{as:u0}, \fer{as:amax}, \fer{as:ad}, \fer{as:ini} such that Theorem \ref{regularity} holds. Let $H$ be the eigenvalue defined in Lemma \ref{eigenHJ}. Then, $\overline u$ (respectively $\underline u$) is a viscosity subsolution (respectively supersolution) of 
\begin{equation}\label{varHJ}
\max (\p_t u- D |\nabla_x u|^2-H,u)=0, \qquad \text{in $(0,\infty)\times\R^d$}.
\end{equation}
\item[(ii)] If we assume additionally \fer{as:ini2}, then $\overline u=\underline u$ and, as $\e$ vanishes, $(u_\e)_\e$ converges locally uniformly to $u=\overline u=\underline u$ the unique viscosity solution of 
\begin{equation*}
\begin{cases}\max (\p_t u- D|\nabla_x u|^2-H,u)=0, \qquad \text{in $(0,\infty)\times\R^d$}, \medskip\\
u(0,x)=u_0(x).
\end{cases}
\end{equation*}
\end{itemize}
\end{prop}

Before proving Proposition \ref{HJlim}, we first give a short proof for Lemma \ref{eigenHJ}.

\begin{proof}[\bf Proof of Lemma \ref{eigenHJ}]

Let $X = \mathcal{C}^{1,\mu}\left( \Theta \right)$ and  $K$ be the positive cone of nonnegative functions in $X$. We define $L:X\to X$ as below 
$$L (u)= - \alpha \Delta_{\theta\theta} u - r \left( a(x,\theta) - a_{\infty} \right) u.$$ 
The resolvent of $L$ together with the Neumann boundary condition is compact from the regularizing effect of the Laplace term. Moreover, the strong maximum principle gives that it is also strongly positive. Using the Krein-Rutman theorem we obtain that there exists a nonnegative eigenvalue corresponding to a positive eigenfunction. This eigenvalue is simple and none of the other eigenvalues correspond to a positive eigenfunction. This defines $H(x)$ and $Q(x,\theta)$ in \fer{ev} in a unique way. The smoothness of $H$ and $Q$  derives from the smoothness of $a(x,\theta)$ and the fact that they are principal eigenelements.

\end{proof}

\begin{proof}[\bf Proof of Proposition \ref{HJlim}.] We prove the result in two steps.\\

\textbf{a. Discontinuous stability (proof of (i)).} \\

To prove the result we need to show that $\overline u$ and $\underline u$ are respectively sub and supersolutions of \fer{varHJ}.\\

\textbf{a.1.} \underline{We prove that $\overline u\leq 0$}. \\

Suppose that there exists a point $(t,x)$ such that $\overline u (t,x) > 0$. From \fer{usup} and \fer{lipv}, there exists a sequence $\eps_n\to 0$ and a sequence of points $(t_n,x_n) \underset{n \to \infty}{\to} (t,x)$ such that, 
\begin{equation*}
u_{\e_n}(t_n,x_n,\theta) \underset{n \to \infty}{\to}  \overline u (t,x), \quad \text{uniformly in $\theta$}.
\end{equation*}
As a consequence, there exists $\delta > 0$ such that for $n$ sufficiently large, $u_{\e_n}(t_n,x_n,\theta) > \delta$, for all $\theta\in \Theta$. This implies
\begin{equation*}
\rho_{\e_n}(t_n,x_n) = \displaystyle \int_{\Theta} \exp\left( \frac{u_{\e_n}(t_n,x_n,\theta)}{\eps_n} \right) d \theta \geq \vert \Theta \vert \exp\left( \frac{\delta}{\eps_n} \right) > a_{\infty},
\end{equation*} 
for sufficiently large $n$, which is in contradiction with Lemma \ref{rhobound}.\\

\textbf{a.2.}  \label{CO} \underline{We prove that $\p_t \overline u- D |\nabla_x \overline u|^2-H\leq 0$}. \\

Now, assume that $\vp \in \mathcal{C}^2 \left( \R^+ \times \R \right)$ is a test function such that $\overline u(t,x)-\vp(t,x)$ has a strict local maximum at $(t_0,x_0)$. 

Using the eigenfunction $Q$ introduced in Lemma \ref{eigenHJ}, we can define a \textit{corrected test function} \cite{LE:89} by $\chi_{\e} = \vp(t,x) + \e \eta(x,\theta)$, with $\eta(x,\theta) = \ln \left( Q\left( x , \theta \right) \right) $. 
Using standard arguments in the theory of viscosity solutions (see \cite{GB:94}), there exists a sequence $(t_\e,x_\e,\theta_\e)$ such that the function $u_\e(t,x,\theta) - \chi_{\e}(t,x,\theta)$ takes a local maximum in $(t_\e,x_\e,\theta_\e)$, which is strict in the $(t,x)$ variables, and such that $(t_\e,x_\e) \to (t_0,x_0)$ as $\e\to 0$. Moreover, as $\theta_{\e}$ lies in the compact set $\Theta$, one can extract a converging subsequence. For legibility, we omit the extraction in the sequel.\\

Let us verify the viscosity subsolution criterion. At the point $(t_\e,x_\e,\theta_\e)$, we have:
\begin{equation*}
\begin{array}{lcl}\p_t \chi_{\e} - D |\nabla_x \chi_{\e}|^2 - H(x_\e) & = &\p_t u_{\e} - D |\nabla_x u_{\e}|^2 - H(x_\e), \medskip \\
&=& \e D\Delta_x u_\e+\f{\al}{\e} \Delta_\theta u_\e +\f{\al}{\e^2}|\nabla_\theta u_\e|^2+r (a(x_{\e},\theta)-\rho_\e) - H(x_\e), \medskip\\
&\leq& \e D\Delta_x \chi_\e + \f{\al}{\e} \Delta_\theta \chi_\e +\f{\al}{\e^2}|\nabla_\theta \chi_\e|^2 + r a(x_{\e},\theta) - H(x_\e).\\
\end{array}
\end{equation*}
We must emphasize that the Neumann boundary conditions are implicitly used here in case when $\theta_{\e}$ is on the boundary of $\Theta$. Indeed, this {ensures} that we have $\nabla_{\theta} \chi_{\e}(t_\e,x_\e,\theta_\e) = 0$ in this latter case. As a consequence, we still have $\nabla_{\theta} u_\e(t_\e,x_\e,\theta_\e) =  \nabla_{\theta} \chi_{\e}(t_\e,x_\e,\theta_\e)$ so that the first order derivative in the trait variable does not add any supplementary difficulty in the r.h.s.. Moreover the second order terms {still have}  the right sign, since again  $\Delta_{xx}^2u_\e \leq \Delta_{xx}^2 \chi_\e$ is enforced {by the Neumann boundary condition}.\\

We replace the test function by its definition to obtain
\begin{multline*}
\p_t \vp - D \vert \nabla_x \left(\vp + \e \eta\right) \vert^2 - H(x_{\e}) \leq \e D \left( \Delta_{xx}^2 \vp + \e \Delta_{xx} \eta \right) + \alpha \left( \Delta_{\theta\theta}^2 \eta + \vert \nabla_{\theta} \eta|^2 \right) + r a(x_{\e},\theta) - H(x_\e).
\end{multline*}

{Here appears the crucial importance of choosing $\eta=\ln Q$ with $Q$ the solution of the spectral problem \fer{ev}. Coupling the above equation with the spectral problem \fer{ev}, written in terms of $\eta$, we deduce that}
\begin{equation*}
\p_t \vp - D \vert \nabla_x \vp + \e \eta \vert^2 - H(x_{\e}) \; \leq \; \e D \left( \Delta_{xx}^2 \vp + \e \Delta_{xx} \eta \right).
\end{equation*}
We conclude, by letting $\e$ go to $0$, that at point $(t_0,x_0)$:
\begin{equation*}
\p_t \vp - D \vert \nabla_x \vp \vert^2 - H \leq 0.  
\end{equation*}

\textbf{a.3.}  \underline{We prove that $\max\left(\p_t \underline u- D |\nabla_x \underline u|^2-H,\underline u\right)\geq 0$}. \\

We first notice that $\underline u(t,x)\leq \overline u(t,x)\leq 0$. Let $\underline u(t,x)<0$. Then there exists some $\delta > 0$ such that along a subsequence $(\e_n,t_n,x_n)$,  
$u_{\e_n}(t_n,x_n,\theta) < - \delta $ for all $\theta\in \Theta$ and for $n\geq N$ with $N$ sufficiently large. It follows that $\rho_{\e_n}(t_n,x_n)\to 0$, as $n\to \infty$. With the same notations as in the previous point replacing maximum by minimum, we get 
\begin{equation*}
\p_t \vp - D \vert \nabla_x \vp + \e \eta \vert^2 - H(x_{\e})  \; \geq \; \e D \left( \Delta_{xx}^2 \vp + \e \Delta_{xx} \eta \right) - r \rho_{\e},
\end{equation*}
so that taking the limit $\eps \to 0$ along the subsequence $(t_{\eps_n}, x_{\eps_n})$, we obtain that 
\begin{equation*}
 \p_t \vp - D |\nabla_x \vp |^2-H \; \geq \; 0.
\end{equation*}
holds at point $(t_0,x_0)$.\\

\textbf{b. Strong uniqueness (proof of (ii)).}\\

Obviously, one cannot get any uniqueness result for the Hamilton-Jacobi equation \fer{varHJ} without {imposing} any initial condition. Adding \eqref{as:ini2}, we now check the initial condition of \eqref{varHJ2} in the viscosity sense.

One has to prove the following
\begin{equation}\label{inicond1}
\min\left(  \max\left(\p_t \overline u-D|\nabla_x \overline u|^2-H,\overline u\right),   \overline u - u_0 \right) \leq 0, \qquad \textrm{ in } \left\{ t = 0 \right\} \times \R^d,
\end{equation}
and
\begin{equation}\label{inicond2}
\max\left( \max\left(\p_t \underline u-D|\nabla_x \underline u|^2-H,\underline u\right),   \underline u - u_0 \right) \geq 0, \qquad \textrm{ in } \left\{ t = 0 \right\} \times \R^d,
\end{equation}
in the viscosity sense.

Here we give only  the proof of \fer{inicond1}, since \fer{inicond2} can be derived following similar arguments. Let $\vp \in \mathcal{C}^2 \left( \R^+ \times \R \right)$ be a test function such that $\overline u(t,x)-\vp(t,x)$ has a strict local maximum at $(t_0 = 0,x_0)$. We now prove that either
\begin{equation*} 
\overline u(0,x_0) \leq u_0(x_0),
\end{equation*}
or 
\begin{equation*} 
\begin{cases}
\p_t \vp (0,x_0) - D|\nabla_x \vp (0,x_0)|^2 - H(x_0) \leq 0, \medskip\\
\textrm{and} \medskip\\
\overline u (0,x_0) \leq 0.
\end{cases}
\end{equation*}

Suppose then that 
\beq
\label{u0x0}
\overline u(0,x_0) > u_0(x_0).
\eeq
 Following the arguments above in \textbf{a.1.}  but taking $t=0$ and using \fer{as:ini2} we obtain 
$$
\overline u (0,x_0) \leq 0.
$$
 We next prove that
\begin{equation}
\label{vpx0} 
\p_t \vp (0,x_0) - D|\nabla_x \vp (0,x_0)|^2 - H(x_0) \leq 0.
\end{equation}
There exists  a sequence $(t_\e,x_\e,\theta_\e)_\e$ such that $(t_\e,x_\e)$  tends to $(0,x_0)$ as $\e\to 0$ and  that $u_\e-\chi_\e=u_\e-\vp-\e\eta$ takes a local maximum at $(t_\e,x_\e,\theta_\e)$. Here $\eta$ still denotes the correction  $\ln Q$ with $Q$ the eigenfunction introduced in Lemma \ref{eigenHJ} (see \textbf{a.1.}). 
We first claim that there exists a subsequence $(t_n,x_n,\theta_n)_n$  of the above sequence and a subsequence  $(\e_n)_n$, with $\e_n\to 0$ as $n\to \infty$, such that
$t_n>0$, for all $n$. \\

Suppose that this is not true. Then, there exists a sequence $(\e_{n'},x_{n'},\theta_{n'})_{n'}$ such that $(\e_{n'},x_{n'})\to (0,x_0)$ and 
that $u_{\e_{n'}}-\vp-\e_{n'}\eta$ has a local maximum at $(0,x_{n'},\theta_{n'})$. It follows that, for all $(t,x,\theta)$ in some neighborhood of $(0,x_{n'},\theta_{n'})$, we have
$$
u_{\e_{n'}} \left( 0,x_{n'},\theta_{n'} \right) - \chi_{\e_{n'}}\left( 0,x_{n'},\theta_{n'} \right) \geq u_{\e_{n'}} \left( t , x , \theta \right) - \chi_{\e_{n'}}\left( t , x , \theta \right).
$$
Computing $ \underset{\underset{(t,x)\to (t_0,x_0)}{n'\to \infty}}{ \overline \limsup}$ at the both sides of the inequality, and using \eqref{as:ini2} one obtains 
 $$
 u_0 \left( x_0 \right) - \vp \left( 0 , x_0 \right) \geq \overline u \left( 0 , x_0 \right) - \vp \left( 0 , x_0 \right).
 $$
However, this is in  contradiction with \fer{u0x0}. We thus proved the existence of subsequences $(t_n,x_n,\theta_n)_n$ and $(\e_n)_n$ described above with  $t_n>0$, for all $n$.\\

Now having in hand that $t_n>0$, from \fer{ue} and the fact that $u_{\e_n}-\vp-\e_n\eta$ takes a local maximum at $(t_n,x_n,\theta_n)$, we deduce that
\begin{equation*}
\p_t \vp - D \vert \nabla_x \vp + \e_n \eta \vert^2 - H(x_{\e_n}) \; \leq  \; \e_n D \left( \Delta_{xx}^2 \vp + \e_n \Delta_{xx} \eta \right)
\end{equation*}
holds in  $(t_n,x_n,\theta_n)$.  Finally, letting $n \to +\infty$, we find \fer{vpx0}.\\

We refer to \cite{GB:94,LE.PS:89} for arguments giving \textit{strong uniqueness} (\textit{i.e.} a comparison principle for semi-continuous sub and supersolutions) for \eqref{varHJ2}. As $\overline u$ and $\underline u$ are respectively sub and supersolutions of \eqref{varHJ2}, we then know that $\overline u \leq \underline u$. From their early definition, we also have $\overline u \geq \underline u$. Gathering these inequalities, we finally obtain $u = \overline u = \underline u$ and {that  $(u_{\e})_{\eps}$ converges locally uniformly, as $\e\to 0$,}  towards $u$, the unique viscosity solution of \eqref{varHJ2} {in} $\R^+ \times \R^d \times \Theta$. 

\end{proof}
%
%
%
%
%
%

\section{Refined asymptotics (The  proof of Theorem \ref{th:main}--(ii) and (iii))}
\label{asympt}

In this section, we provide some information on the asymptotic population density. Firstly, we prove parts (ii) and (iii) of Theorem \ref{th:main} which state that the zero sets of $u$ correspond to the zones where the population is positive. Secondly, we provide the limit of $(n_\e)_\e$, as $\e\to 0$,  in a particular case (see Proposition \ref{convvarsep}). \\

We first prove the following {l}emma:
\begin{lem}\label{grate}
Let $u$ be the unique viscosity solution of \fer{varHJ2} and $H(x)$ the eigenvalue given by  Lemma \ref{eigenHJ}. Then 
\begin{equation*}
(t,x) \in \text{Int} \left \lbrace u(t,x) = 0 \right\rbrace \quad \Longrightarrow \quad H(x) \geq 0.
\end{equation*}
\end{lem}

\begin{proof}[\bf Proof of Lemma \ref{grate}]
Thanks to \eqref{varHJ2}, 
\begin{equation*}
\partial_t u - D \vert \nabla_x u \vert^2 \leq H(x),
\end{equation*}
in the viscosity sense. In the zone $\text{Int} \left \lbrace u(t,x) = 0 \right\rbrace$, one has
\begin{equation*}
\partial_t u - D \vert \nabla_x u \vert^2 = 0,
\end{equation*}
in the strong sense. The proof of the lemma follows.
\end{proof}

We now are able to characterize the different zones of the front and complete the proof of Theorem \ref{th:main}:

\begin{proof}[\bf Proof of Theorem \ref{th:main}, (ii)]

Let  $K$ be a compact subset of $\text{Int}\left\lbrace u < 0\right\rbrace$.  The local uniform convergence of $u_\eps$ towards $u$ ensures that there exists a constant $\delta > 0$ such that for sufficiently small $\eps > 0$ and for all $(t,x)\in K$ and $\theta\in \Theta$, $u_\eps(t,x,\theta) < - \delta$. As a consequence, $n_\eps = \exp\left( \frac{u_\eps}{\eps} \right)  < \exp\left( - \frac{\delta}{\eps} \right)  \to 0 $, uniformly as $\eps \to 0$ {in} $K\times \Theta$. \\
\end{proof}

\begin{proof}[\bf Proof of Theorem \ref{th:main}, (iii)]

 Take $(t_0,x_0) \in K \subset \subset \text{Int} \left( \left\lbrace u = 0\right\rbrace \cap \left\lbrace H(x) > 0 \right\rbrace \right)$, and let $Q$ be the normalized eigenvector given by Lemma  \ref{eigenHJ}. We denote $C_m = C_m(x_0)=\min_{\Theta} Q(x_0,\theta)$ and $C_M=C_M(x_0) = \max_{\Theta} Q(x_0,\theta)$. We also define \begin{equation*}
F_\eps(t,x) := \int_\Theta n_\eps(t,x,\theta) Q(x_0,\theta) d \theta,\qquad I_\e:=\e\ln F_\e.
\end{equation*}
Using classical arguments, one can prove that $I:=\lim_{\e\to 0}I_\e$ is well-defined and nonpositive.
We also point out that $\{u=0\}=\{I=0\}$ since
\begin{equation*}
\min_\theta u_\e( t,x,\theta) + o(1) \leq  I_\eps(t,x) \leq \max_\theta u_\e( t,x,\theta) +o(1),
\end{equation*} 
for all $(t,x)\in [0,\infty)\times \R^d$. Multiplying equation \eqref{main2} by $Q(x_0,\theta)$ and integrating in $\theta$ yields
\begin{equation*}
\eps \partial_t F_\eps - \eps^2 D\Delta_x F_\eps - \al\int_\Theta n_\eps \Delta_\theta Q(x_0,\theta) =r \int_\Theta a(x,\theta)Q(x_0 , \theta) n_\eps (t,x,\theta) d\theta - r \rho_\eps F_\eps.
\end{equation*}
Combining the above equation by \fer{ev} we deduce that
\begin{equation*}
\eps \partial_t F_\eps - \eps^2D \Delta_x F_\eps = \left( H(x_0) -r \rho_\eps \right) F_\eps +r \int_\Theta Q(x_0 , \theta) n_\eps (t,x,\theta) \left[ a(x,\theta) - a(x_0,\theta) \right] d\theta.
\end{equation*}
Since $H(x_0)>0$ and $a$ is continuous,  for all $\delta > 0$, one can choose a constant $r>0$ such that 
\begin{equation*}
\forall x \in B_r(x_0), \qquad \vert  a(x,\theta) - a(x_0,\theta)  \vert < \delta H(x_0)\quad \text{with $B_r(x_0) \subset K$}.
\end{equation*}
We finally deduce that for all $ x \in B_r(x_0)$,
\begin{equation*}
\eps \partial_t F_\eps - \eps^2D \Delta_x F_\eps \geq \left( (1 - \delta)H(x_0) - r \rho_\eps \right)F_\eps.
\end{equation*}
Since $F_\eps \geq C_m \rho_\eps$, it follows that  for all $ x \in B_r(x_0)$,
\begin{equation}\label{superF}
\eps \partial_t F_\eps - \eps^2 D \Delta_x F_\eps \geq \left( (1 - \delta)H(x_0) - \frac{ rF_\eps}{C_m} \right)F_\eps.
\end{equation}
Moreover, since $(t_0,x_0)\subset \mathrm{Int}\{ u(t,x)=0\}=\mathrm{Int}\{ I(t,x)=0\}$, we have $I(t,x)=0$ in a neighborhood of $(t_0,x_0)$.\\

We then apply an argument similar to the one used in \cite{LE.PS:89} to prove an analogous statement for the Fisher-KPP equation.
To this end, we introduce the following test function
\begin{equation*}
\varphi(t,x) =- \vert x - x_0 \vert^2 - (t-t_0)^2.
\end{equation*}
As $I-\varphi$ attains a strict minimum in $(t_0,x_0)$, there exists a sequence $(t_\eps,x_\eps)$ of points such that and 
$I_\e-\varphi$ attains a minimum in $(t_\e,x_\e)$, with $(t_\eps,x_\eps) \to (t_0,x_0)$. It follows that
\begin{equation*}
\partial_t \varphi - \eps D\Delta_x \varphi - D\vert \nabla_x \varphi \vert^2 \geq \partial_t I_\eps - \eps D\Delta_x I_\eps - D\vert \nabla_x I_\eps \vert^2 = (1 - \delta)H(x_0)-\frac{rF_\eps}{C_m} .
\end{equation*}
As a consequence, 
\begin{equation*}
\liminf_{\eps \to 0} F_\eps(t_0,x_0) \geq \f{C_m}{r} (1 - \delta)H(x_0),
\end{equation*}
uniformly with respect to points $(t_0,x_0)\in K$ and this gives 
\begin{equation*}
\liminf_{\eps \to 0} \rho_\eps (t_0, x_0) \geq (1-\delta) H(x_0) \frac{C_m}{rC_M}.
\end{equation*}
We then let $\delta \to 0$ and obtain 
\begin{equation*}
\liminf_{\eps \to 0} \rho_\eps (t_0, x_0) \geq  H(x_0)  \frac{C_m(x_0)}{rC_M(x_0)},
\end{equation*}
uniformly with respect to points $(t_0,x_0)\in K$.\\

Let 
$$
\widetilde K=\{x\;|\; \exists t \geq 0, \text{ such that} \;(t,x)\in K\}.
$$
To conclude the proof, it is enough to prove that there exists a constant ${\overline C}={\overline C}(\al,r,a\vert_{\widetilde K\times \Theta})\geq 1$, such that
$$ \frac{C_M(x)}{C_m(x)}\leq {\overline C}, \qquad \text{for all $x\in \widetilde K$}.$$
This is indeed a consequence of the Harnack inequality \cite{JB.MS:04} for the solutions of \fer{ev} in $\Theta$ for all $x\in \widetilde K$. 
We point out that here we can use the Harnack inequality on the whole domain $\Theta$ thanks to the Neumann boundary condition.
\end{proof}

%

The above result is not enough to identify the limit of $(n_\e)_\e$ as $\e\to 0$, as was obtained for example for Fisher-KPP type models in \cite{LE.PS:89}. The main difficulties to obtain such limits are the facts that we do not have any regularity estimate in the $x$ direction on $n_\e$ and that there is no comparison principle for this model due to the non-local term. However, we were able to identify the limit of  $(n_\e)_\e$ in a particular case:\\

\begin{prop}\label{convvarsep}
Suppose that $Q$, the eigenvector given by \eqref{ev}, does not depend on $x$, i.e. $Q(x,\theta)=Q(\theta)$. Let the initial data be of the following form
\begin{equation}
\label{propin}
n_{\eps}(t=0,x,\theta) = m_{\eps}(x)Q(\theta), \qquad m_\e(x)\geq 0.
\end{equation}
Then:
\begin{itemize}
\item[(i)] For all $t>0$ and $(x,\theta) \in \R^d \times \Theta$, \quad $n_{\eps}(t,x,\theta) = m_{\eps}(t,x)Q(\theta).$
\item[(ii)] For all $(t,x,\theta) \in \left\lbrace u(t,x) = 0 \right\rbrace \times \Theta$, \quad $\lim_{\eps \to 0} n_\eps(t,x,\theta) = \frac{H(x)}{r} Q(\theta)$.
\end{itemize}

\end{prop}

\begin{rmk}
We note that the assumption on $Q$ in Proposition \ref{convvarsep}, is satisfied for $a(x,\theta)=a(\theta)+b(x)$.
\end{rmk}

\begin{proof}[ \bf Proof of Proposition \ref{convvarsep}]

Let $m_\e$ be the unique solution of the following equation
$$
\begin{cases}
\e \p_t m_\e - \e^2 D \Delta_x m_\e =r m_\e \left( H(x)- m_\e\right),\\
m_\e(0,x)=m_\e(x).
\end{cases}
$$
Define 
$$
\widetilde n_\e(t,x,\theta) : = m_\e(t,x) Q(\theta).
$$
We notice from \fer{normQ} that
$$
\int \widetilde n_\e(t,x,\theta) d\theta= m_\e(t,x).
$$
Consequently, from \fer{propin}, \fer{ev}, and the definition of $m_\e$ one can easily verify that $\widetilde n_\e$ is a solution of \fer{main2}, and since \fer{main2} has a unique solution we conclude that
$$
n_\e(t,x,\theta)= m_\e(t,x) Q(\theta),
$$
and 
$$
\rho_\e(t,x)= m_\e(t,x). 
$$
As a consequence $\rho_\eps$ satisfies the following Fisher-KPP equation
\begin{equation*}
\eps \partial_t \rho_\eps - \eps^2 D \Delta_x \rho_\eps = r \rho_\eps \left( H(x) - \rho_\eps \right) . 
\end{equation*}
Let $(t,x) \in \left\lbrace u = 0 \right\rbrace$. Then from Lemma \ref{grate}, we obtain $H(x)\geq 0$.
Hence, from the above equation and \fer{rhoinf}, following similar arguments as in \cite{LE.PS:89} (page 157) we obtain that $\rho_\e(t,x)\to H(x)$ as $\e\to 0$, and (ii) follows.
\end{proof}

\section{Qualitative properties}\label{Qualitative}

In this {s}ection, we provide some estimates on the effective Hamiltonian $H$ and the eigenfunction $Q$. We note that  the spatial propagation of the population can be described using $H$ through \fer{varHJ2}. In particular, if $H(x)=H$ is constant and if initially the population is restricted to a compact set in space, then the population propagates in space with the constant speed  $c=2\sqrt{H}$. Furthermore, the eigenfunction $Q$ is expected to represent the phenotypical distribution of the population (see Proposition \ref{convvarsep}).\\

We begin by presenting some qualitative estimates on the effective Hamiltonian $H$.
\begin{lem}\label{estH}
The eigenvalue and normalized eigenfunction introduced in {Lemma \ref{eigenHJ}} satisfy the following estimates:
\beq
\label{H1}
\forall x \in \R, \quad H(x) =r \int_{\Theta} a(x,\theta) Q(x,\theta) d \theta,
\eeq
\begin{equation}
\label{H2}
\forall x \in \R, \quad \frac{r}{\Theta} \int_{\Theta} a(x,\theta) d \theta \leq H(x) \leq ra(x, \bar \theta (x) )
\end{equation}
where $\bar\theta$ is a trait which maximizes $Q(x,\cdot)$: $Q(x,\bar \theta (x)) = \max_{\Theta} Q(x,\theta)$.
\end{lem}

In particular, the eigenvalue $H(x)$, which more or less represents the speed of the front, is not necessarily given by the most privileged individuals, that is those having the largest fitness $a$. See Example \ref{Schro} for a case where the inequality is strict. This property confirms that the front may slow down due to very unfavorable traits. 
\begin{proof}[\bf Proof of Lemma \ref{estH}.]

By integrating \fer{ev} with respect to $\theta$ and using the Neumann boundary condition and \fer{normQ}, we find \fer{H1}.\\

To prove \fer{H2} we rewrite \fer{ev} in terms of $\eta = \ln Q$:
\begin{equation}\label{eqeta}
\forall (x,\theta) \in \R \times \Theta, \qquad H(x) = \alpha \left( \Delta_{\theta\theta} \eta + \vert \nabla_{\theta} \eta|^2 \right) + r a(x,\theta)\end{equation}
Then, integrating and using the Neumann boundary conditions in the variable $\theta$ for $\eta$, one obtains
\begin{equation*}
H(x) \geq \frac{r}{\vert \Theta \vert} \int_{\Theta} a(x,\theta) d \theta.
\end{equation*}
Let $Q(x,\bar \theta (x)) = \max_{\Theta} Q(x,\theta)$. Then $\nabla_{\theta} \eta (x,\bar \theta(x)) = 0$ and $\Delta_{\theta} \eta (x, \bar \theta(x)) \leq 0$.  Evaluating \eqref{eqeta} in $\bar \theta(x)$, we get
\begin{equation*}
H(x) \leq ra(x, \bar \theta (x) ).
\end{equation*}
\end{proof}

\begin{lem}\label{conczero}
Let $a(x,\cdot)$ be a strictly concave function on $ \Theta := [\theta_m , \theta_M]$ for all $x \in \R$. Then for all $x \in \R$, the maximum of $Q(x,\cdot)$ is attained in only one point $\bar \theta (x)$.
\end{lem}

\begin{proof}[\bf Proof of Lemma \ref{conczero}]
The concavity hypothesis implies that for all $x \in \R$, the function $H(x) - ra(x,\cdot)$ is strictly convex. Thus, on the interval $[\theta_m , \theta_M]$, it has at most two zeros. The case of no zeros is excluded from \fer{H2}. Let's study the two remaining cases. 

Suppose it has only one zero {at} $\hat\theta$ (see Example \ref{toads}), say it is positive on $\left[ \theta_m , \hat \theta \right[ $ and nonpositive on $\left[ \hat \theta , \theta_M \right[ $. Then from the early definition of the spectral problem, $Q(x,\cdot)$ is convex on $\left[ \theta_m , \hat \theta \right[ $ and concave on $\left[ \hat \theta , \theta_M \right[ $. The Neumann boundary conditions enforce that $Q(x,\cdot)$ is increasing on $\Theta$, and attains its maximum at $\bar \theta = \theta_{M}$.

Suppose it has two zeroes, {at}  $\hat\theta_1$ and  $\hat\theta_2$. Then $H(x) - r a(x,\cdot)$ is nonnegative on $\left[ \theta_m , \hat \theta_1 \right[ $ and $\left[ \hat  \theta_2 , \theta_M \right[ $, and negative on $\left[ \hat \theta_1 , \hat \theta_2 \right[ $. As a consequence, by the same convexity analysis as in the previous case, $Q(x,\cdot)$ attains its maximum on  $\left[ \hat \theta_1 , \hat \theta_2 \right[ $, where it is strictly concave, which justifies the existence and uniqueness of $\bar \theta$.

\end{proof}
\begin{rmk}{{\bf (Limit as $\al\to 0$).}}\label{alphazero}

As the mutation rate $\alpha$ goes to $0$, we expect that the eigenfunction $Q_\al$ converges towards a sum of Dirac masses. To justify this, we use again a WKB ansatz, setting $\varphi_\alpha = \sqrt{\alpha} \ln \left( Q_\alpha \right)$. Rewriting \fer{ev} in terms of $\varphi_\alpha$ we obtain
\begin{equation*}
\sqrt{\alpha} \Delta_\theta \varphi_{\alpha}+\vert \nabla_\theta \varphi_\al \vert^2 + r a(x,\theta) - H_\alpha(x)  = 0.
\end{equation*}
It is classical that the family $\vp_\al$ is equi-Lipschitz and we can extract a subsequence that converges uniformly. We have indeed that as $\al\to 0$, $(\vp_\al,H_\al)$ converges to $(\vp,H)$, with $\vp$ a viscosity solution of the following equation
\begin{equation*}
\begin{cases}
\vert \nabla_\theta \varphi \vert^2  + r a(x,\theta) - H(x)  = 0,\medskip\\
H(x)=\max_{\theta\in \Theta} ra(x,\theta).
\end{cases}
\end{equation*}
Moreover from \fer{normQ} we obtain that 
\begin{equation*}
\max_{\theta \in \Theta} \vp(x,\theta)=0.
\end{equation*}
Finally, we conclude from the above equations that as $\al\to 0$, $Q_\al \,{\xrightharpoonup{\quad} }  \;Q$ with $Q$ a measure satisfying 
\begin{equation*}
\mathrm{supp}\; Q(x,\cdot)\subset \{\theta\in\Theta \,\vert\, \vp(x,\theta)=0 \} 
\subset \{\overline \theta\in\Theta \,\vert\, H(x)={r}a(x,\overline \theta)={r}\max_{\theta\in \Theta} a(x,\theta) \}.
\end{equation*}
In other terms, in the limit of rare mutations, the population concentrates on the maximum points of the fitness $a(x,\theta)$.
\end{rmk}

\section{Examples and numerics}\label{Examples}
\subsection{Examples of spectral problems}

In this section, we present various spectral problems to discuss the properties of the principal eigenfunction $Q$ depending on the form of the fitness $a$. The principal eigenfunction $Q$ is expected, at least in some cases, to represent the asymptotic phenotypic distribution of the population (see Proposition \ref{convvarsep}). The examples are illustrated in Tables 1 and 2.

\begin{table}[h]\label{table1}
\begin{center}
\begin{tabular}{| l | c | c | c |}
\hline
  & Exemple 1 &  {Exemple 2 ($\theta_m=0.5$)} & {Exemple 2 ($\theta_m= 0.75$)}\\
  \hline
  (1) & \includegraphics[width = 0.3\linewidth]{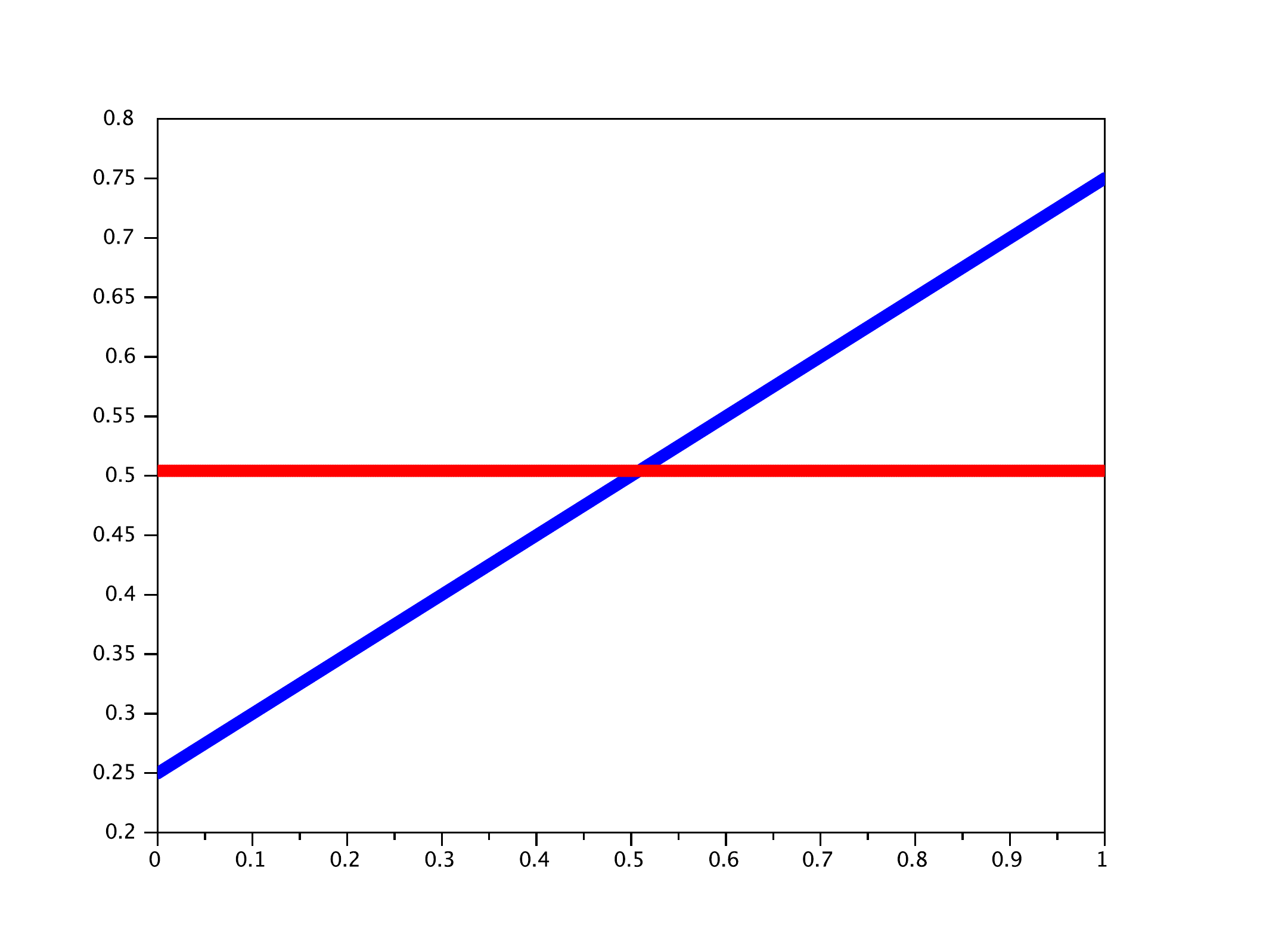}
 & 
\includegraphics[width = 0.3\linewidth]{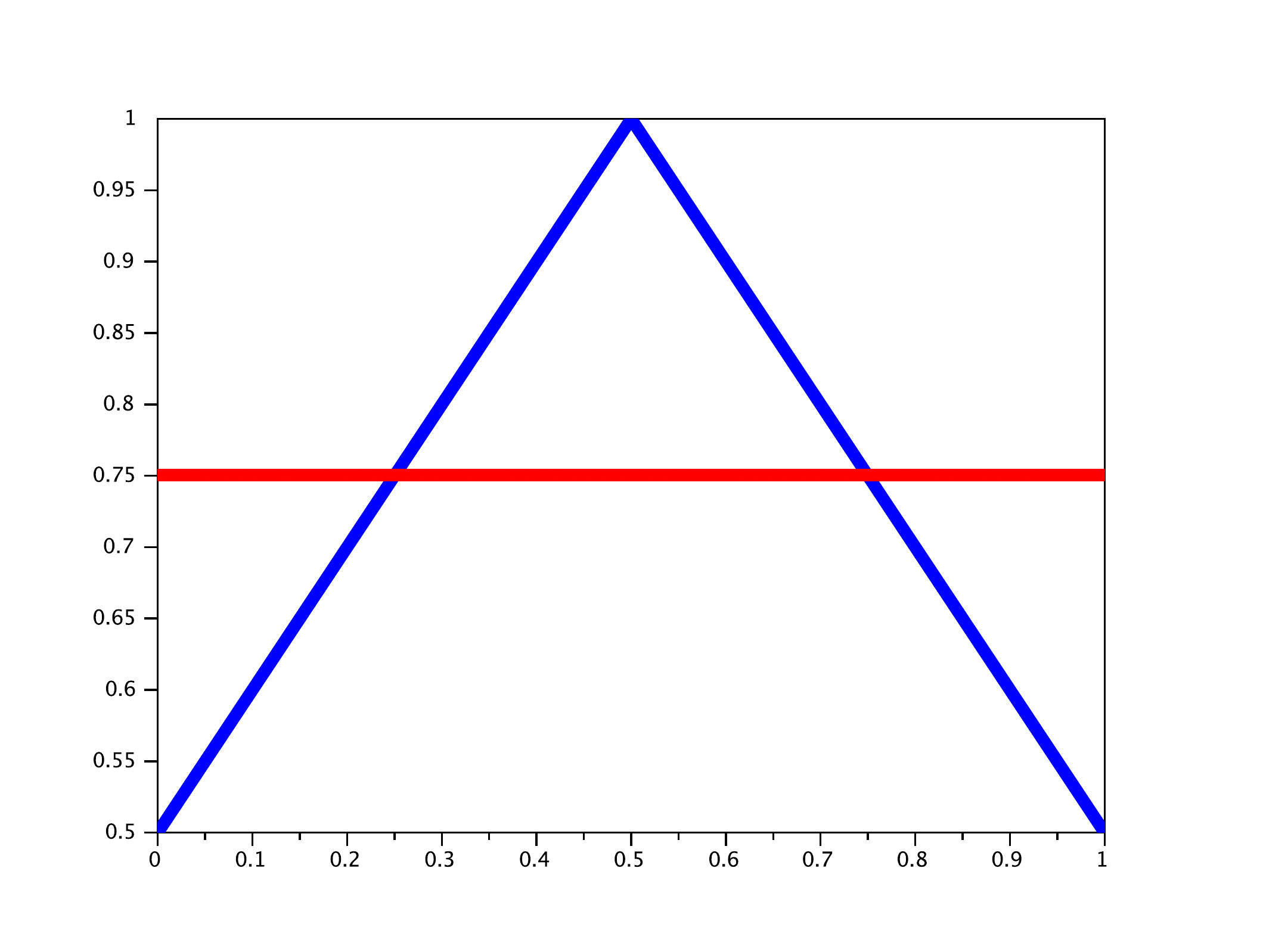} 
  & 
\includegraphics[width = 0.3\linewidth]{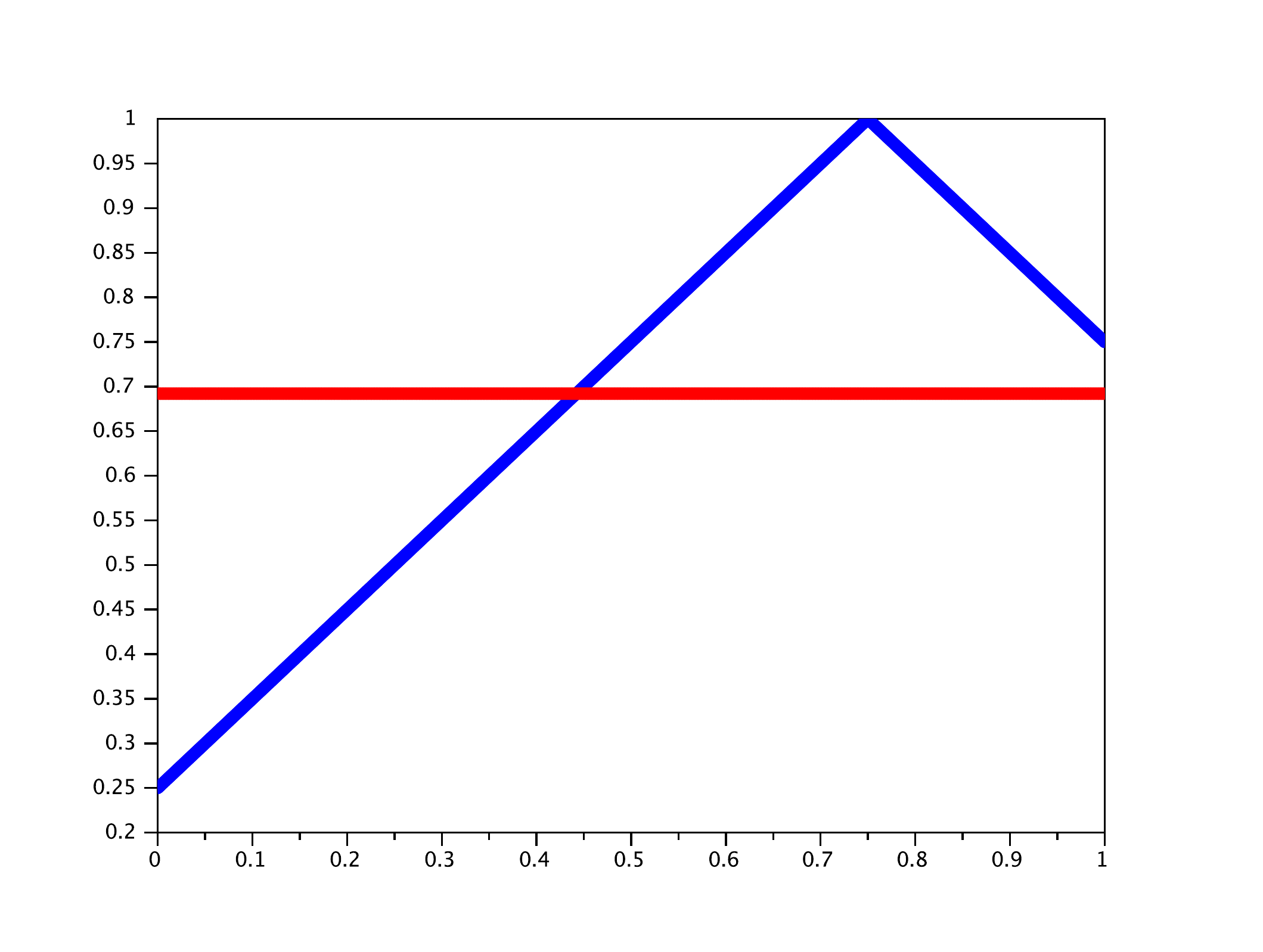}\\
  \hline
  (2) & 
  \includegraphics[width = 0.3\linewidth]{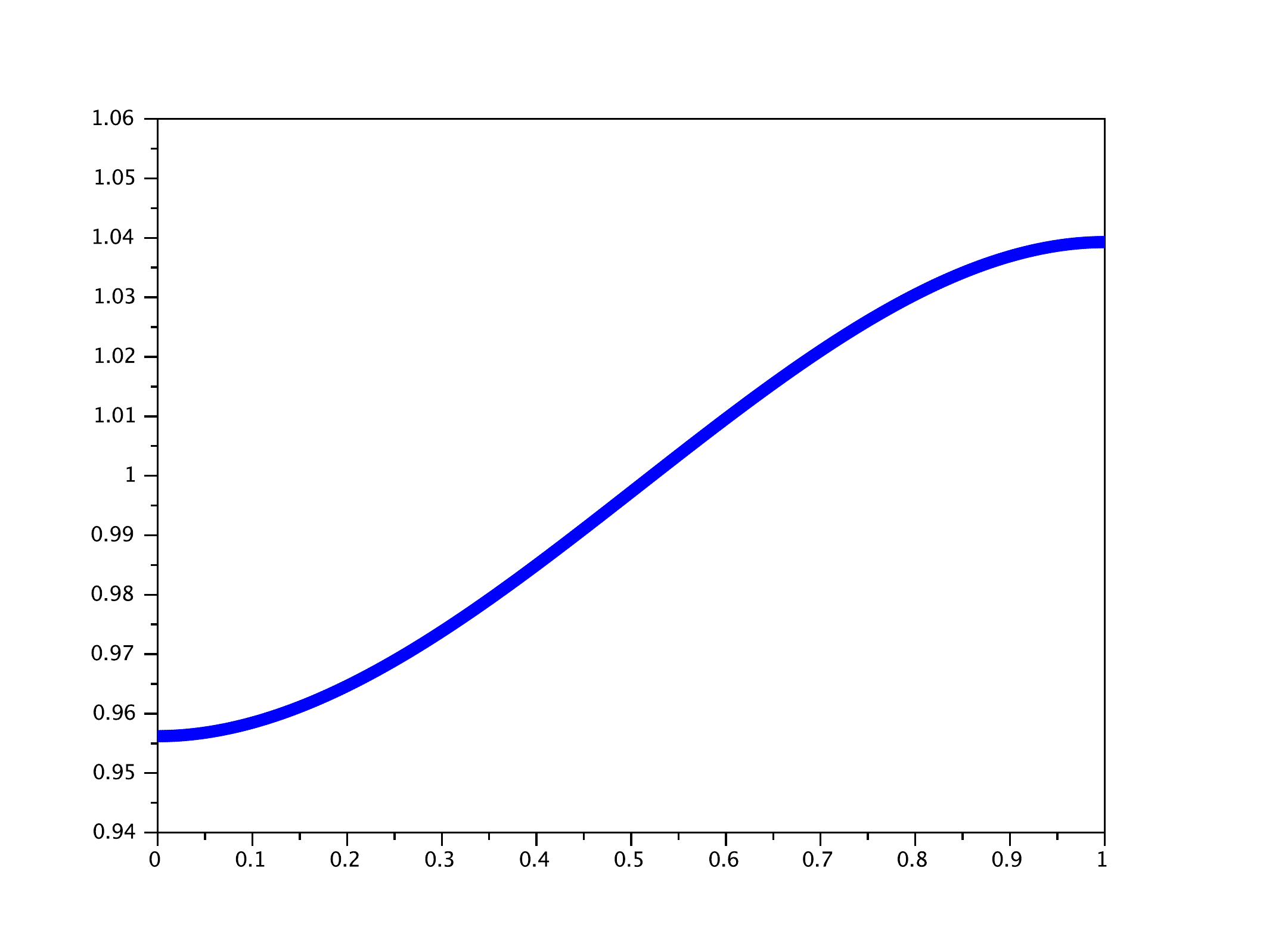}
& \includegraphics[width = 0.3\linewidth]{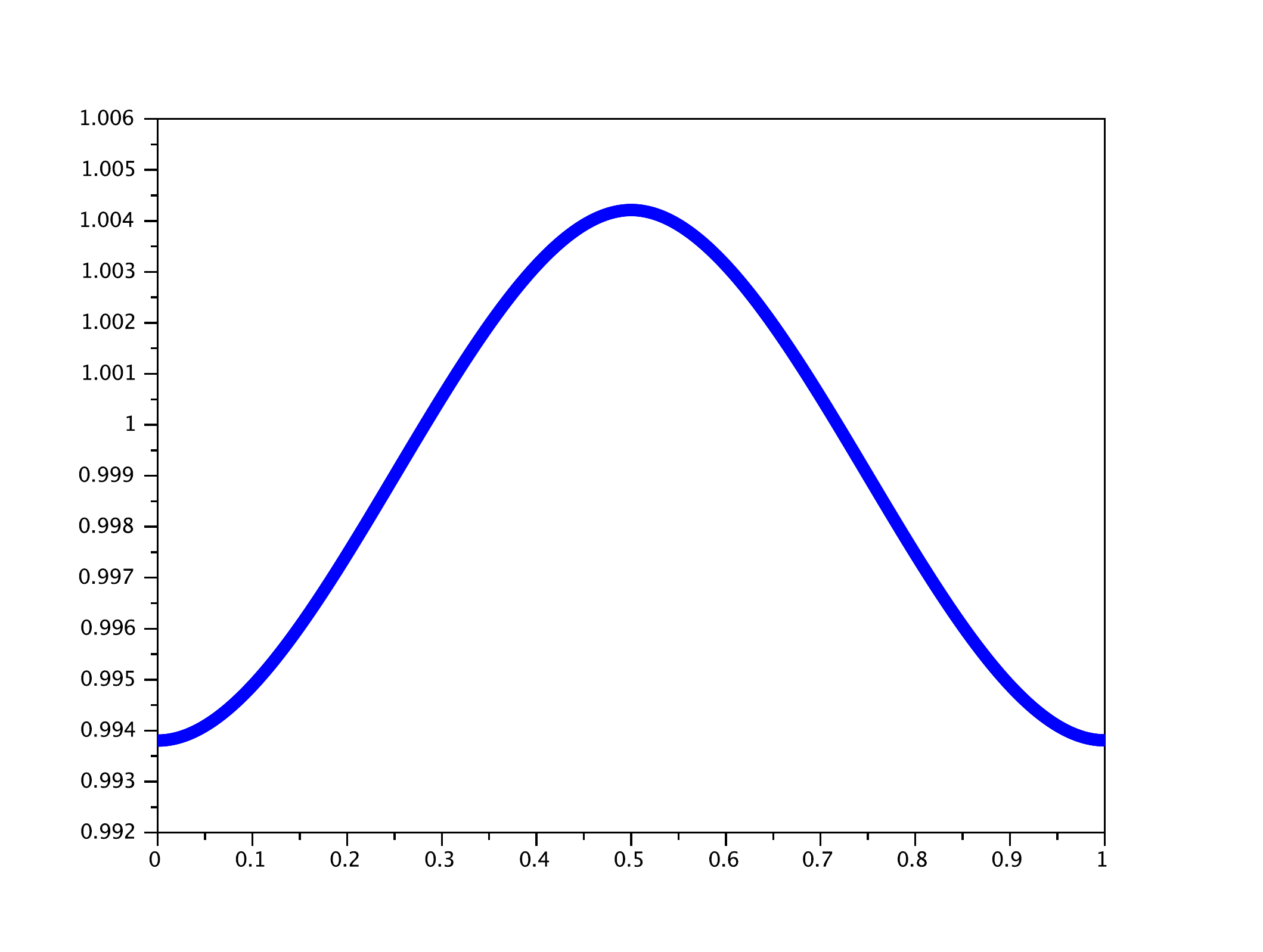} 
& 
\includegraphics[width = 0.3\linewidth]{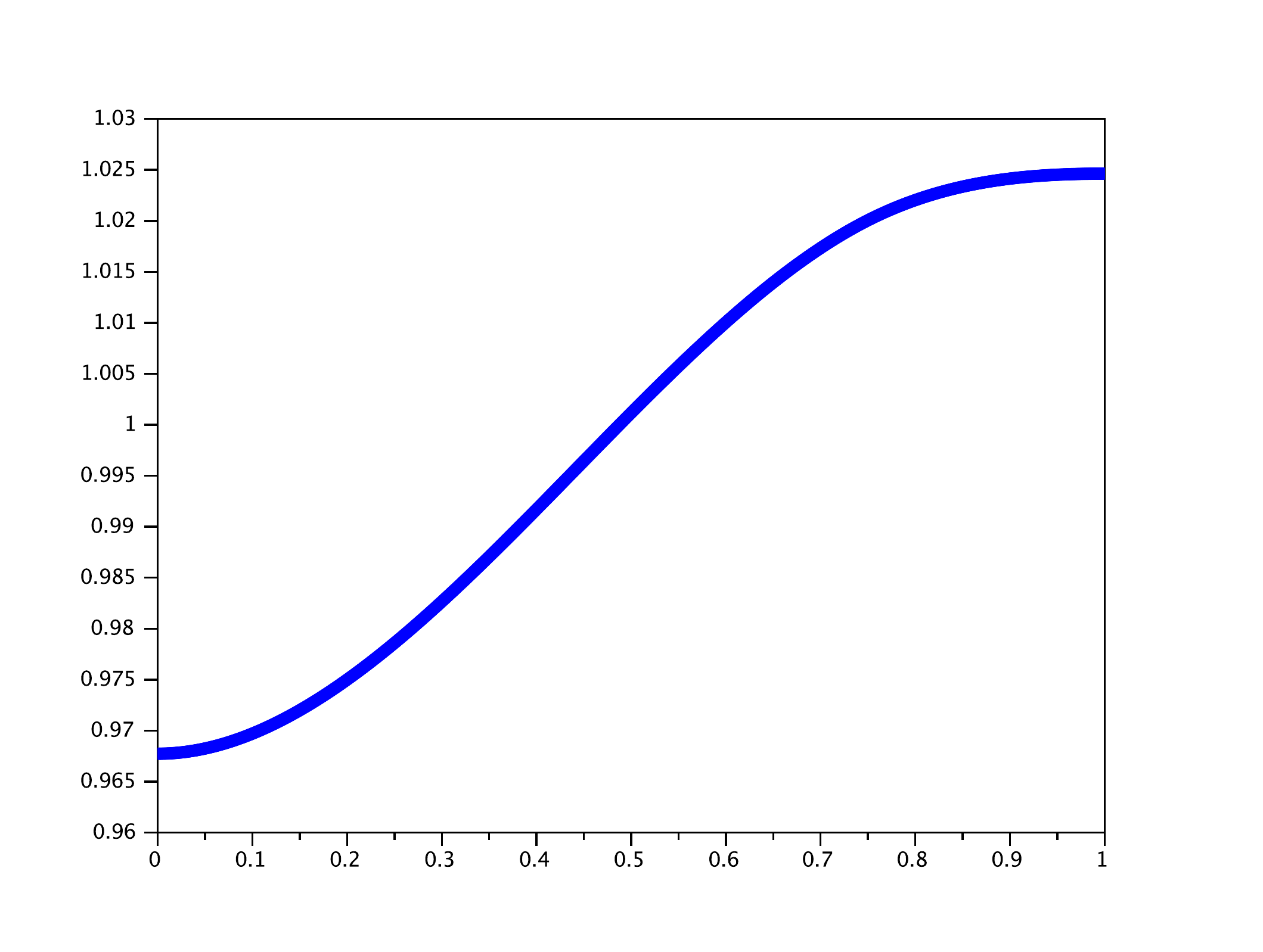}
\\
  \hline
\end{tabular}
\end{center}
\caption{(1): Fitness $a$ (in blue) and principal eigenvalue $H$ (in red); (2): Renormalized principal eigenfunction $Q$.}
\end{table}

\begin{ex}{\bf (A fitness with linear dependence {on} $\theta$).} \label{toads}

This example is taken from  \cite{EB.VC.NM.SM:12}. For $\theta \in \left[ \theta_m , \theta_M \right]$ and $b : \R \to \Theta$ a smooth function, let $a(x,\theta) = \mu \theta - b(x)$. The spectral problem writes, for all $x \in \R$:
\begin{equation*}
\begin{cases}
\alpha \partial_{\theta\theta} Q + r \theta Q = \left( H(x) + r b(x) \right) Q, \medskip\\
\partial_{\theta} Q (x , \theta_m) = \partial_{\theta} Q (x , \theta_M) = 0.
\end{cases}
\end{equation*} 

The solution of this problem is unique up to a multiplicative constant and can be expressed implicitly with special Airy functions. In Table 1, for $\alpha = 1$, $r = 2$ and $\Theta =[0,1]$, we plot the fitness $a(\theta) = \frac{\theta}{2} + \frac{1}{4}$ and the associated eigenvector $Q$. Beware that this example will be used again in Section \ref{num}.

\end{ex}

\begin{ex}{\bf (The maxima of $a$ and $Q$ are not always at the same points.)}

For Example {2}, we consider $a(x,\theta) := 1 - \vert \theta - \theta_m\vert$, for different values of $\theta_m \in \Theta := [0,1]$. The parameters for the simulations are $\alpha = 1, r = 1$. {We observe that although the fitness $a$ attains its maximum at $\theta = \theta_m$, it is not given that the maximum of the eigenfunction Q is attained at $\theta = \theta_m$. In other words, the  trait  with the optimal fitness value does not necessarily correspond to the most represented one.} Indeed, when $\theta_m = \frac12$, the eigenfunction $Q$ is necessarily symmetric with respect to $\theta_m=\frac12$, and hence attains a maximum at this point. By contrast, for $\theta_m = \frac34$, the most represented trait is not the most favorable one {(see Table 1)}: The diffusion through the Neumann boundary condition plays a strong role in this case.  {We observe indeed with this example that, while the fitness $a$ has a {non-symmetric} profile,  the maximum points of $Q$ can be far from the ones of $a$, due to the diffusion term. However, while $\al$ (which equals $1$ in this example) takes values close to $0$, the maximum points of $Q$ approach the ones of  $a$.}

\end{ex}

\begin{ex}{\bf (An example of $a$ and $Q$ with two maximum points)}

For this example, we consider $a(x,\theta) := \varphi_i( \theta )$, for $i=1,2$, and  $\varphi_i$  a quartic function such that two different traits are equivalently favorable in the population. Nevertheless, $Q$ can still take a single maximum on a different point. First, we consider the following symmetric fitness function:
\begin{equation*}
\varphi_1 (\theta):= 200 \left( \theta - \frac15 \right)\left( \theta - \frac25 \right)\left( \theta - \frac35 \right)\left( \theta - \frac45 \right),
\end{equation*}
which has two maxima but all traits between the two maxima are also likely to survive. It turns out that the mutation plays a strong role and creates a single peak in $Q$, which is necessarily $\frac12$ by symmetry. In the second case, we consider the fitness function
\begin{equation*}
\varphi_2 (\theta):= 100 \left( \theta - \frac19 \right)\left( \theta - \frac13 \right)\left( \theta - \frac23 \right)\left( \theta - \frac89 \right),
\end{equation*}
which is still symmetric with respect to the center of $\Theta$. However, since there is a gap between the two traits with the most optimal fitness value, the eigenfunction $Q$ has also two peaks but at different points. See Table 2 for the different plots ($\alpha = 1, r = 1$).
\end{ex}
\begin{table}[h]
\begin{center}
\begin{tabular}{| l | c | c | }
\hline
   & {Exemple 3 ($\varphi_1$)} & {Exemple 3 ($\varphi_2$)} \\
  \hline
  (1) 
& 
\includegraphics[width = 0.3\linewidth]{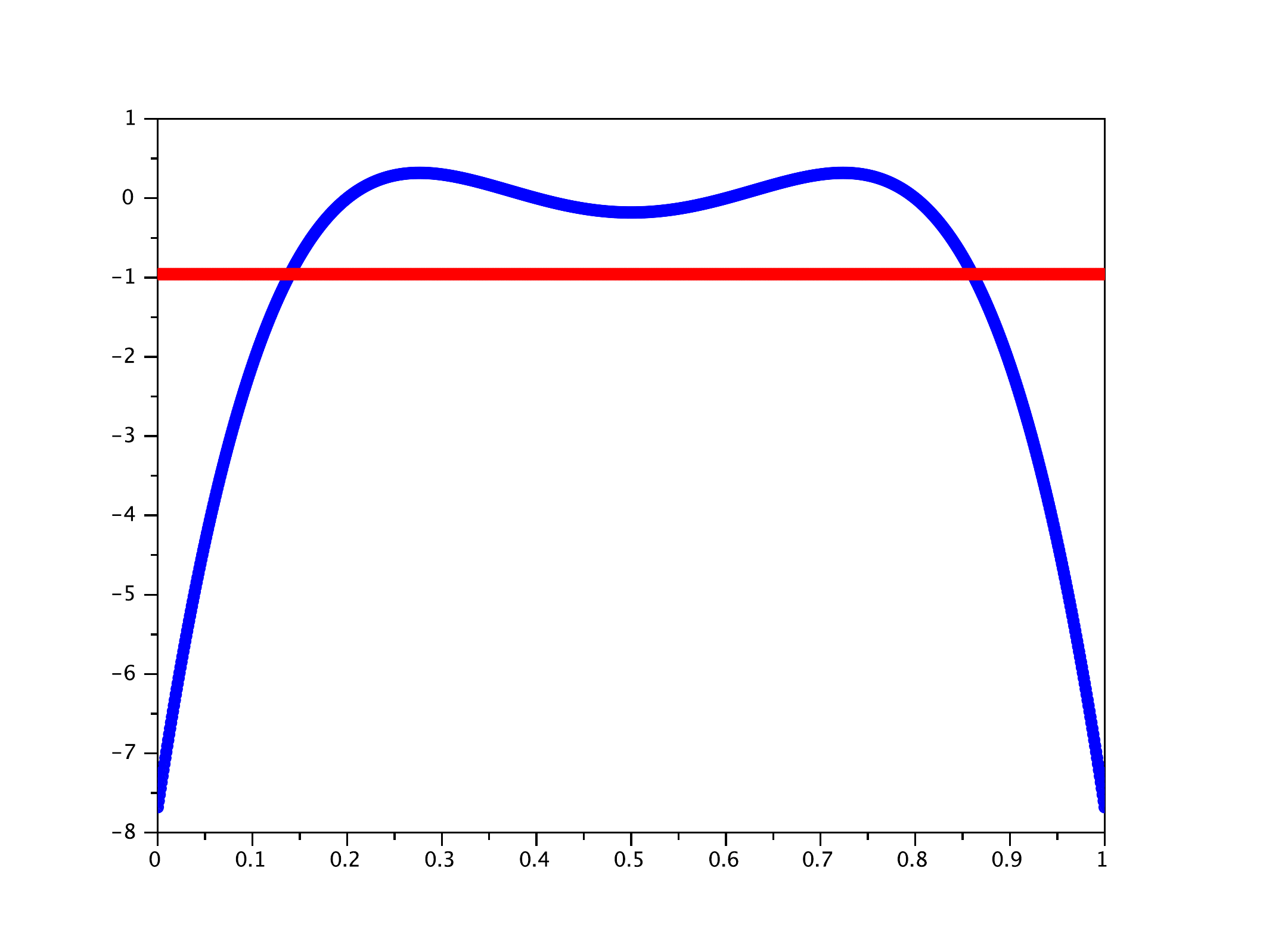}
& 
\includegraphics[width = 0.3\linewidth]{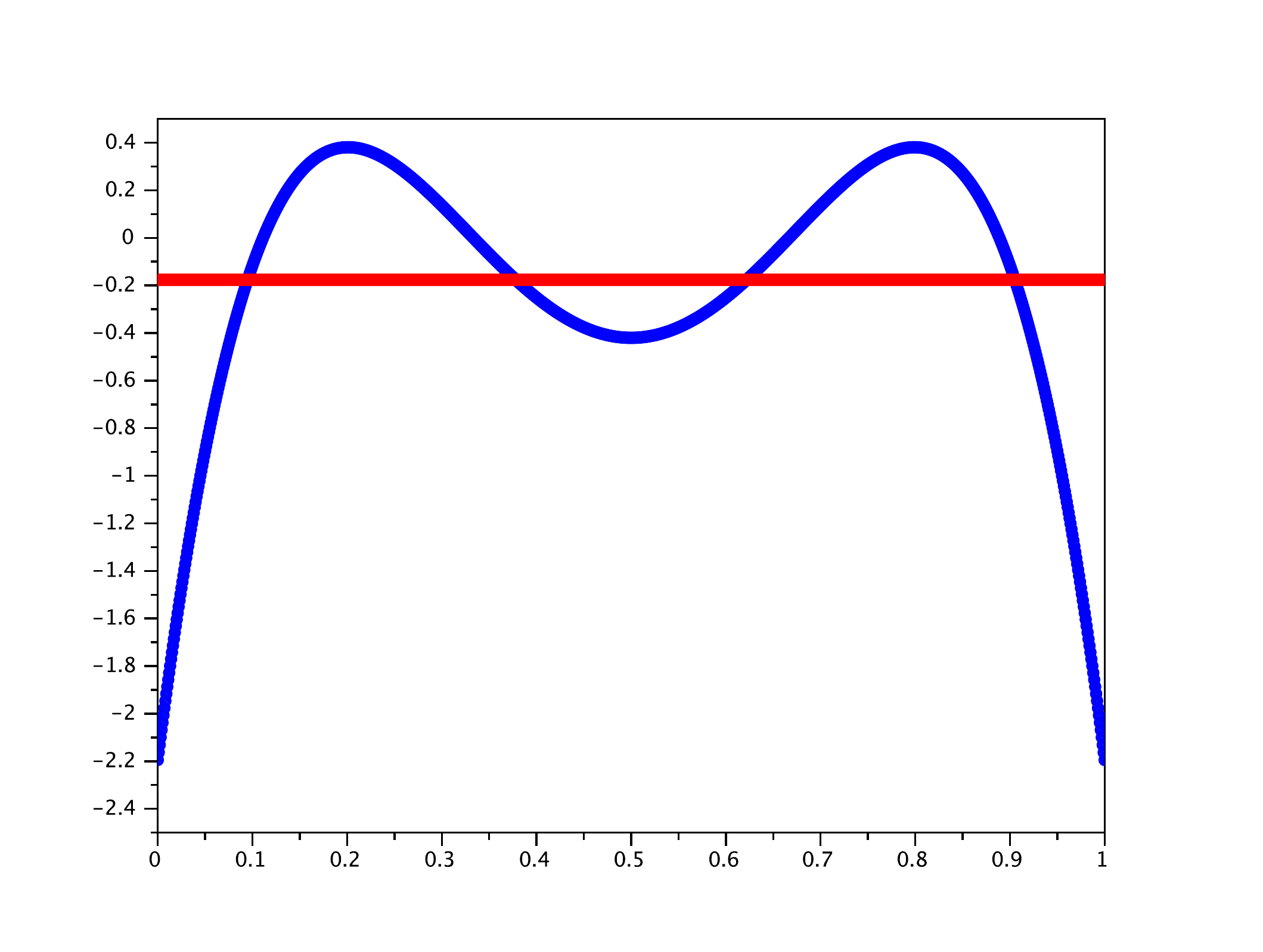} \\
  \hline
  (2) 
& \includegraphics[width = 0.3\linewidth]{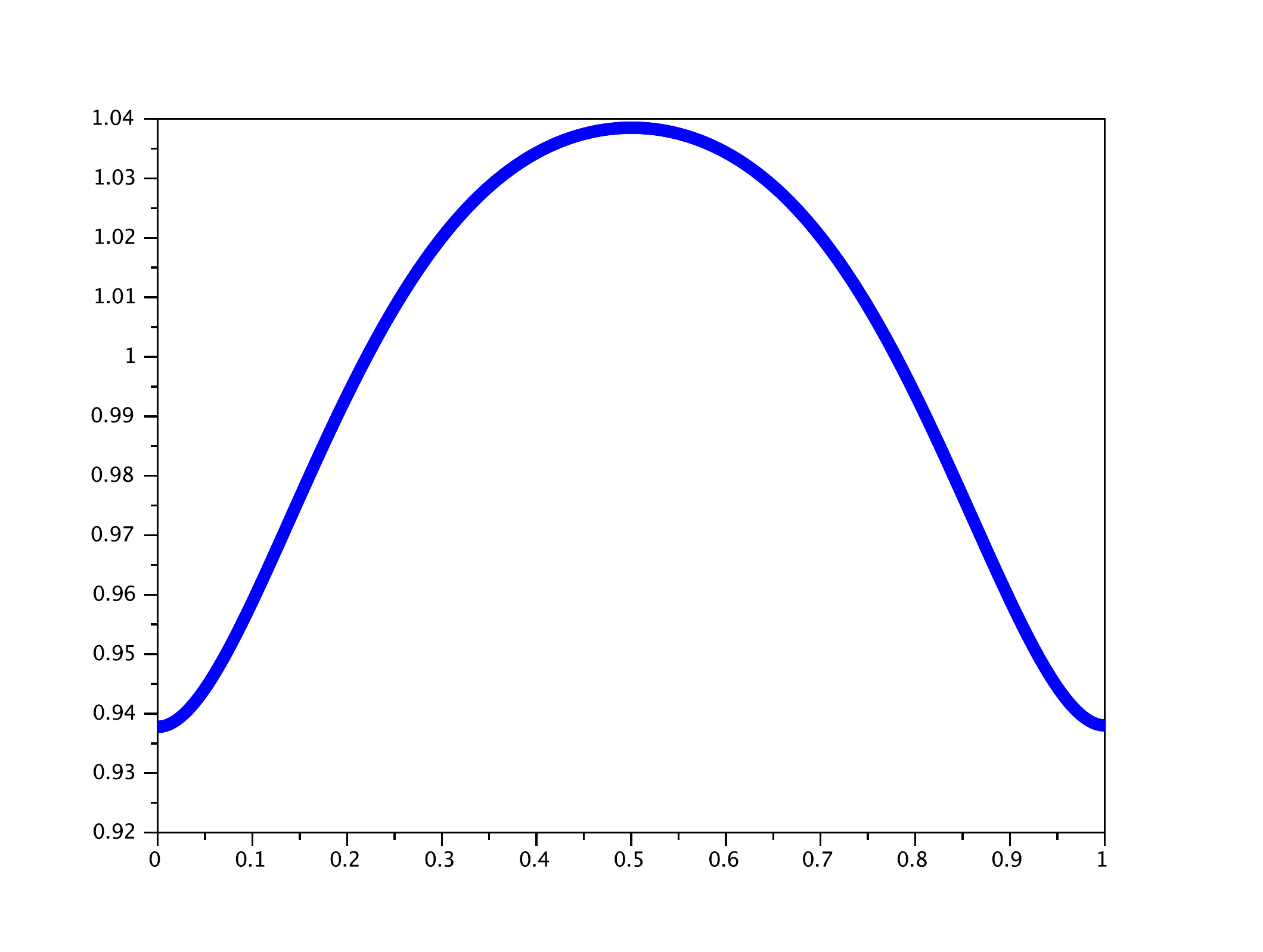}
& \includegraphics[width = 0.3\linewidth]{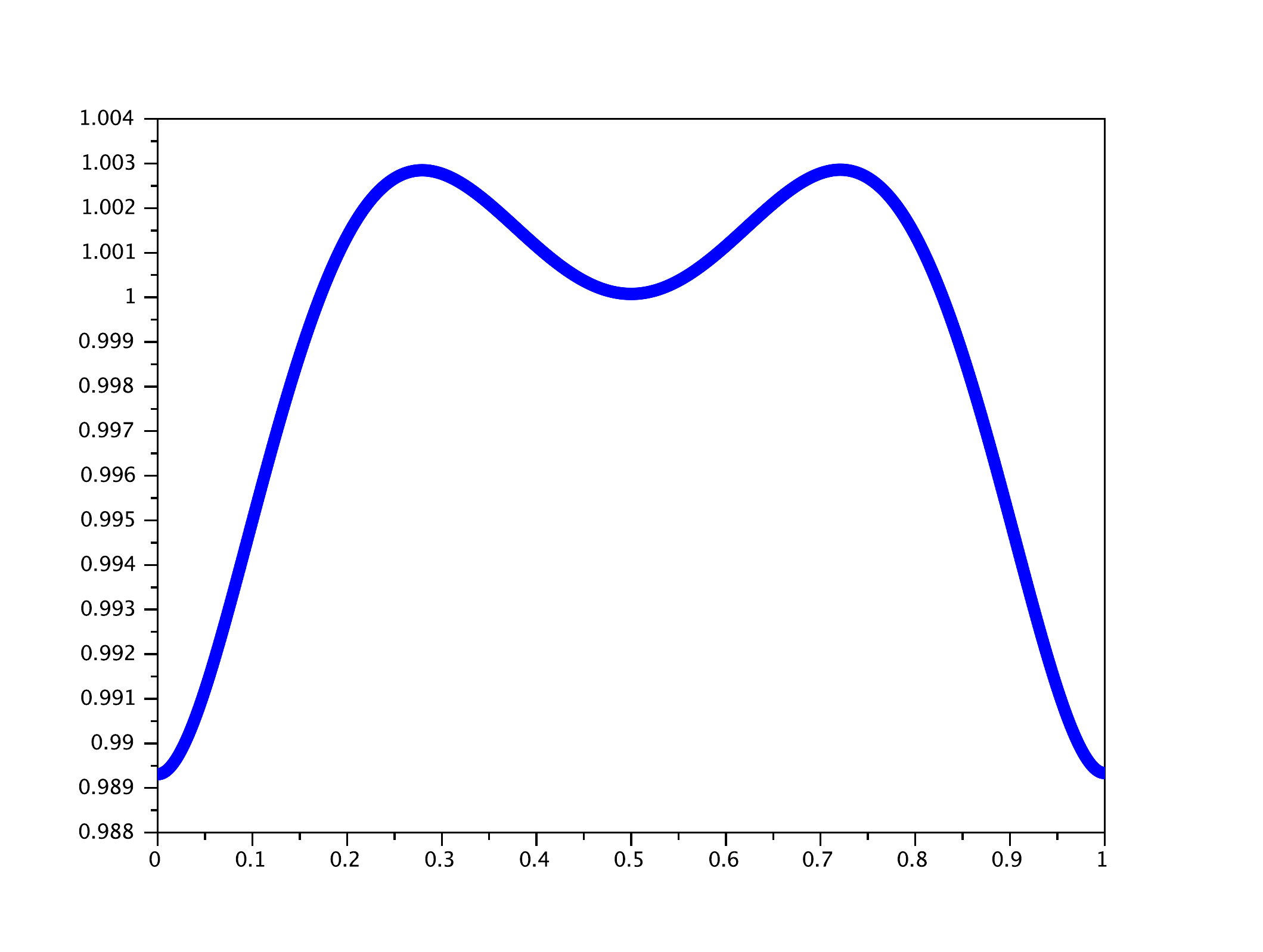}
 \\
  \hline
\end{tabular}
\end{center}
\caption{Fitness $a$ (in blue) and principal eigenvalue $H$ (in red); (2): Renormalized principal eigenfunction $Q$.}
\end{table}

\begin{ex}[{\bf An example with $\Theta$ unbounded}]\label{Schro}

Although not within the framework of this article, we expect that under coercivity conditions on $-a$, Theorem \ref{HJlim} would be still true with an unbounded domain $\Theta$, and in particular for $\Theta=\R^d$. Here, we give an example with $\Theta=\R$ for which  it is easy to compute the eigenelements $Q$ and $H$. We consider $a\left( x, \theta \right) := a_{\infty} - \frac{b_{\infty}}{2}\left( \theta - b(x) \right)^2$ where $b: \R \to \R$ is a smooth function. We can then compute:
\begin{equation*}
Q(x,\theta) =\exp\left( - \frac12 \sqrt{ \frac{r b_{\infty} }{2 \alpha} } \left( \theta - b(x) \right)^2 \right), \qquad H(x) = r a_{\infty} - \sqrt{ \frac{r \alpha b_{\infty} }{2} }.
\end{equation*}
These solutions are not valid in a bounded domain since they do not satisfy the Neumann boundary conditions. \\

We note that, with these parameters, the left inequality in \fer{H2} is strict, i.e. $H(x) < ra(x, \bar \theta (x) )=ra_\infty$,
but as $b_{\infty} \to 0$, which corresponds to the limit case where all the traits are equally favorable, we have $H(x) \to r a_{\infty}$. Finally, it is interesting to notice here that to ensure front expansion, \textit{i.e.} $H(x) \geq 0$, the fitness must satisfy the following additional condition $\frac{2 a_{\infty}^2 }{b_{\infty}} \geq \frac{\alpha}{r}$. 


\end{ex}

\subsection{Numerical illustrations on the evolution of the front}\label{num}

In this section, we resolve numerically the evolution problem \fer{main2} for three different values of $a$. For all the examples, we choose the following initial data
\begin{equation}
\label{iniex}
\begin{cases}
n^\eps(0,x ,\theta) = \max \left( 1 , 2 - 8 \left( \theta - \frac12 \right)^2 \right)& \text{if $x \in \left[0,0.5\right]$}, \theta \in \Theta := [0,1],\\
n^\eps(0,x ,\theta) =0&\text{otherwise}.
\end{cases}
\end{equation}
and use the following parameters
\begin{equation}
\label{par}
\alpha = 1, \qquad r = 2, \qquad D=1,\qquad  \eps = 0.1. 
\end{equation}
The numerical simulations have been performed in Matlab. We gather our results in Figures 1 - 2 - 3. 
For the three different fitness functions, we plot, from left to right:
\begin{itemize}
\item[(+)] The density $n^\e(t,x,\theta)$ for a given final time $t=T$, 
\item[(+)] The value of $\rho^\eps(x)$ at this same final time (blue line), that we compare to the value of $\max \left( \frac{H(x)}{r} , 0 \right)$ (red line), 
\item[(+)] The renormalized trait distributions at the edge of the front (red square-shaped line) and at the back (blue star-shaped line) that we compare to the expected renormalized eigenfunctions $Q$ at the same space positions (pink {circle-shaped} lines).
\end{itemize}
The fitness functions used in the three figures, are respectively
$$a_1(\theta) =  \frac{1}{4} + \frac{\theta}{2},$$
$$a_2(x,\theta) = a_1(\theta) + \left( \sin(x) - \frac12 \right),$$
and
$$a_3(x,\theta) = a_1(\theta) \left( 1 + \frac{1}{1 + 0.05\, x^2} \right).$$
For the three examples, we observe propagation in the $x$--direction as expected according to Theorem \ref{th:main}. We also notice that, in the zones where the front has arrived, i.e. in the set $\mathrm{Int}\{u=0\}$, $\rho^\e$ converges to $\max(\frac{H}{r},0)$. Moreover, for $\e$ small, the renormalized trait distribution of the population at position $x$, i.e. $\f{n_\e(t,x,\cdot)}{\int n_\e(t,x,\theta'd\theta')}$, is close to $Q(x,\cdot)$. These properties have been proved theoretically for a particular case in Proposition \ref{convvarsep}. We also notice that the convergence of the averaged density $\rho^\e$ seems to be faster than the convergence of the density $n^\eps$. \\

In Figure 2, we illustrate an example where $H$ is periodic in $x$ and it can take negative values. This corresponds to a case where the population faces some obstacles, i.e. zones where the conditions are not favorable for the population to persist. However, according to the numerical illustrations, the population manages to pass through the obstacles and reach the favorable zones where it can grow up again. Indeed, even if asymptotically as $\e\to 0$ the density $n_\e$ goes to $0$ in these harsh zones, in the $\e$--level, $n_\e$ is positive but exponentially small. This small density can reach the better zones and grow up.  See Figures 1, 2 and 3 for detailed comments.

%

\begin{figure}[h]\label{Numconveps}
\begin{center}

\includegraphics[width = 0.32\linewidth]{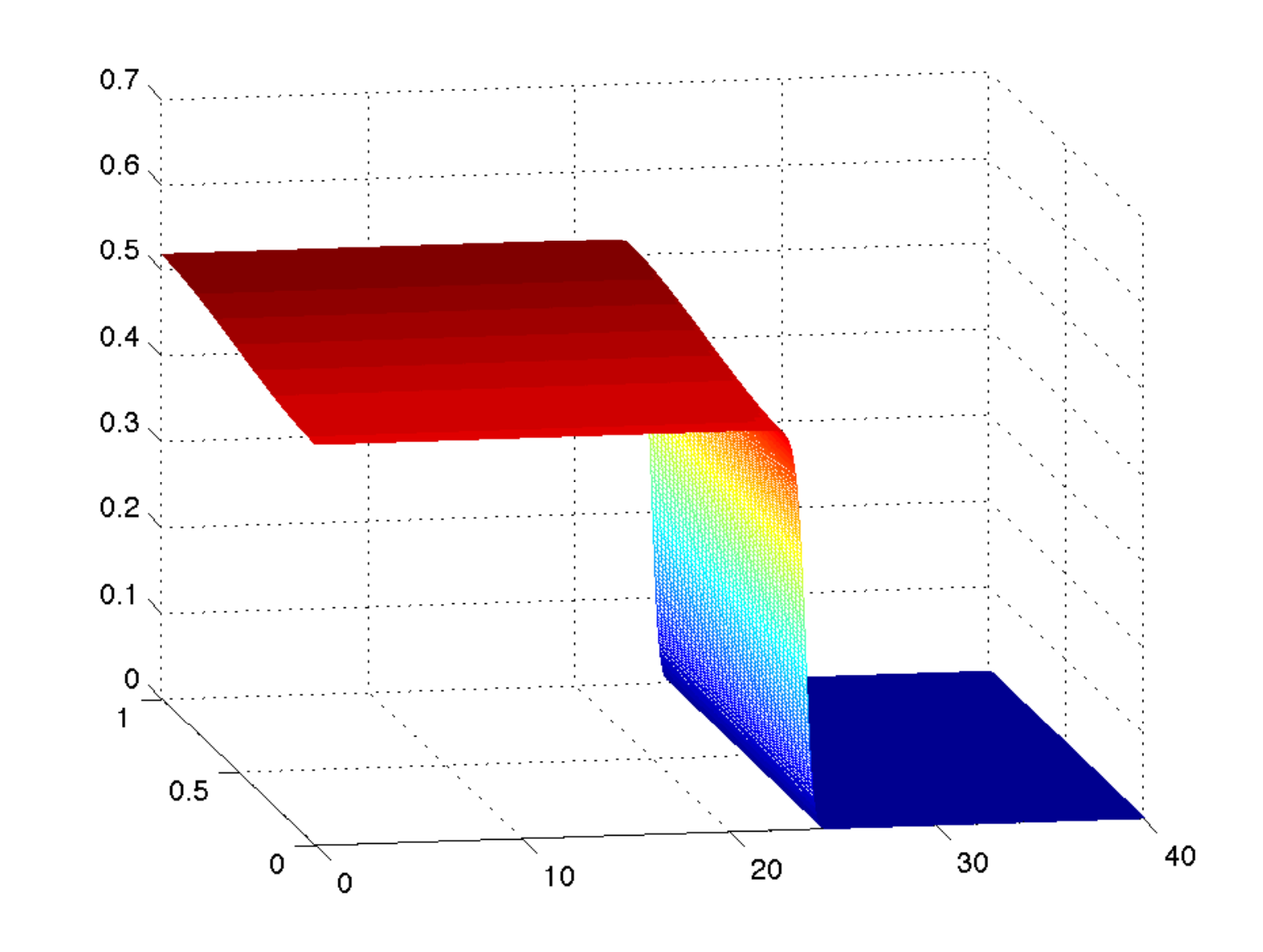}
\includegraphics[width = 0.32\linewidth]{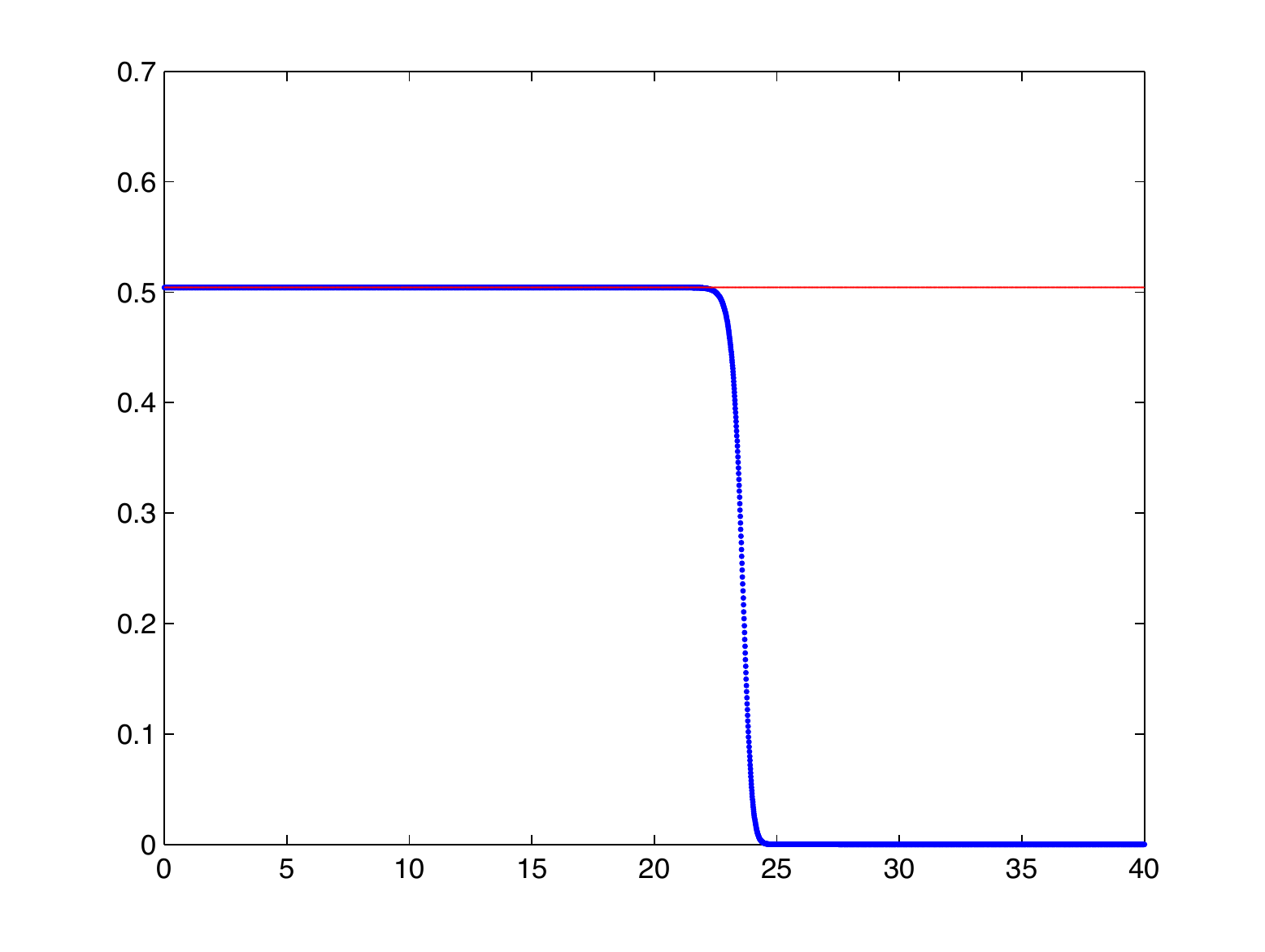}
\includegraphics[width = 0.32\linewidth]{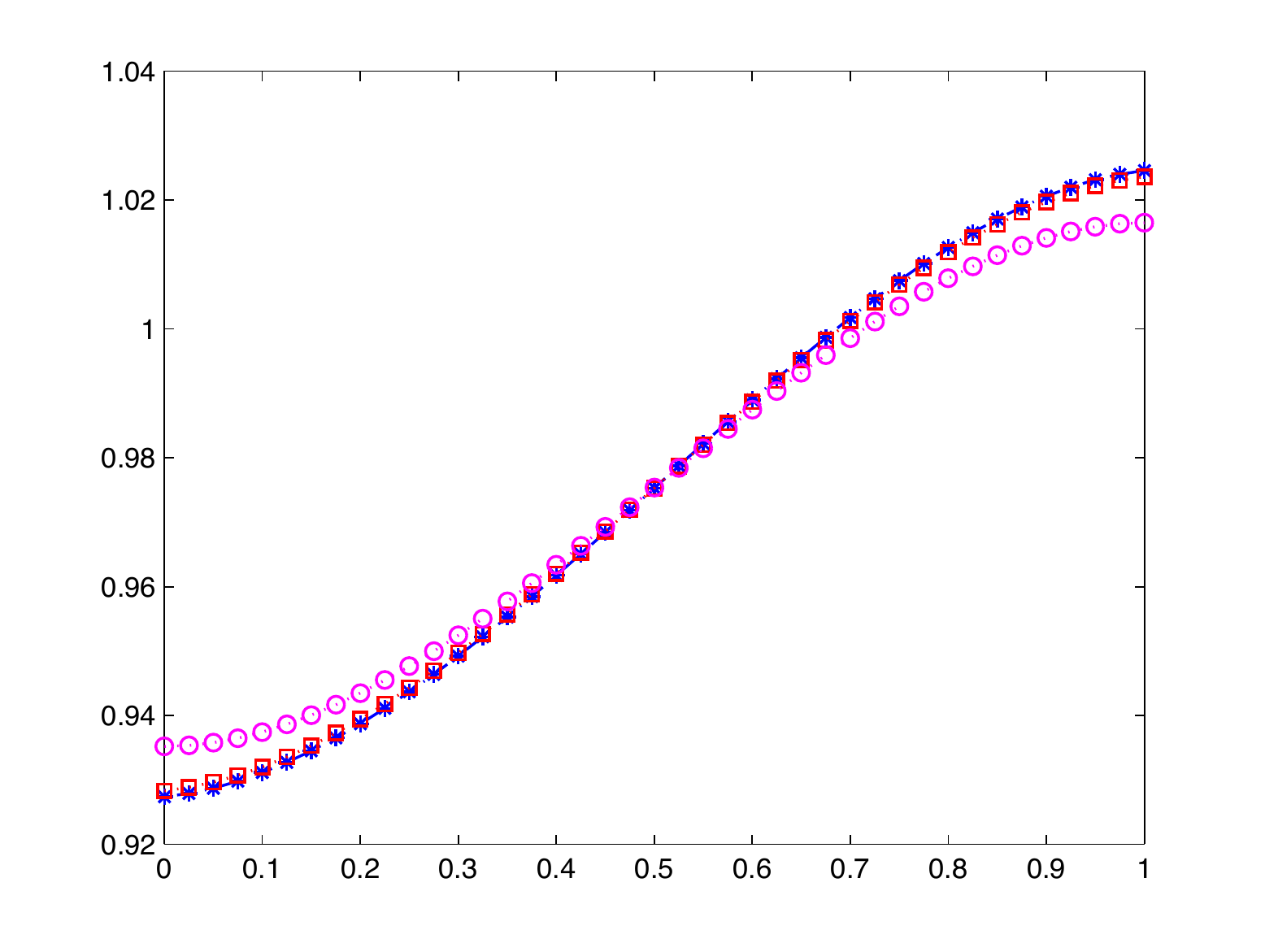}

\caption{We present numerical resolution of \eqref{main2} with the fitness $a_1(\theta) =  \frac{1}{4} + \frac{\theta}{2}$ and using the initial data and the parameters  given by \fer{iniex} and \fer{par}. In this case, as expected, $\rho^\e$ converges to $\frac{H}{r}$ in the zone where the front has arrived: in the set $\{u=0\}$. We also observe that the renormalized trait distribution at the edge (red square-shaped line) and the back of the front (blue star-shaped line), are close to the principal eigenfunction $Q$ (pink circle-shaped line), noting that $Q$ here does not depend on $x$. These results are in accordance with Proposition \ref{convvarsep}.
}
\end{center}
\end{figure}

\begin{figure}[h]\label{Numconveps2}
\begin{center}

\includegraphics[width = 0.32\linewidth]{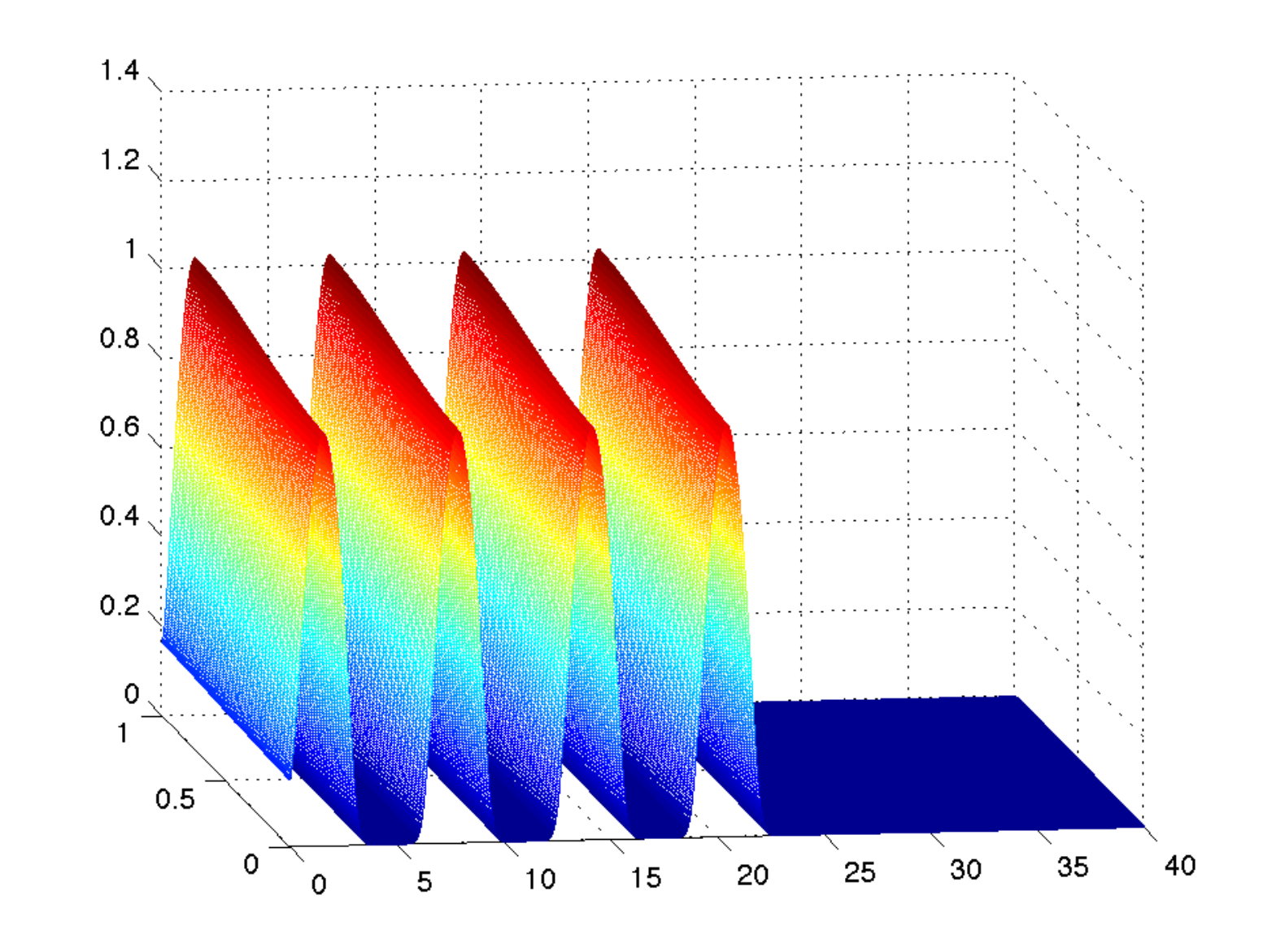}
\includegraphics[width = 0.32\linewidth]{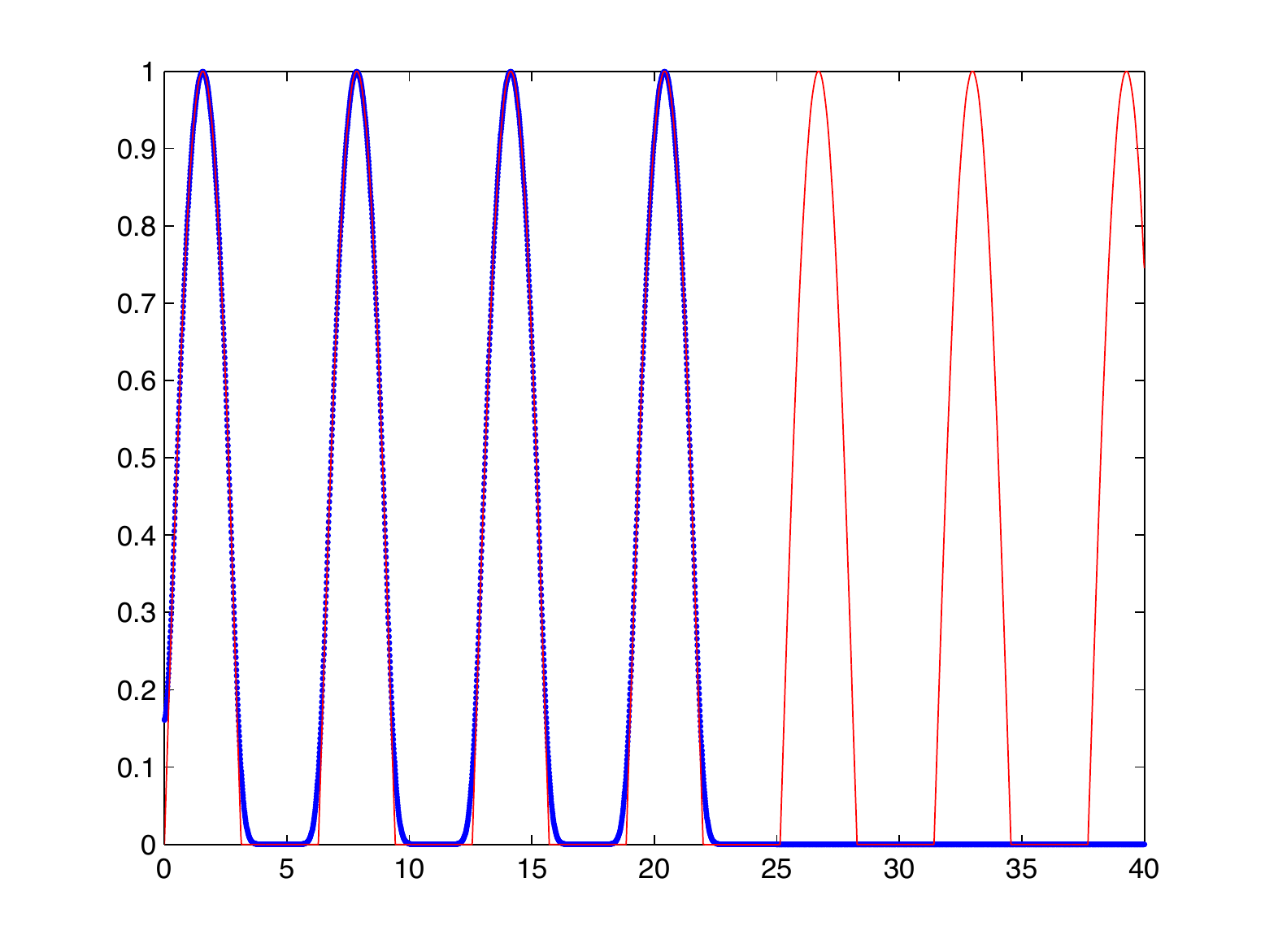}
\includegraphics[width = 0.32\linewidth]{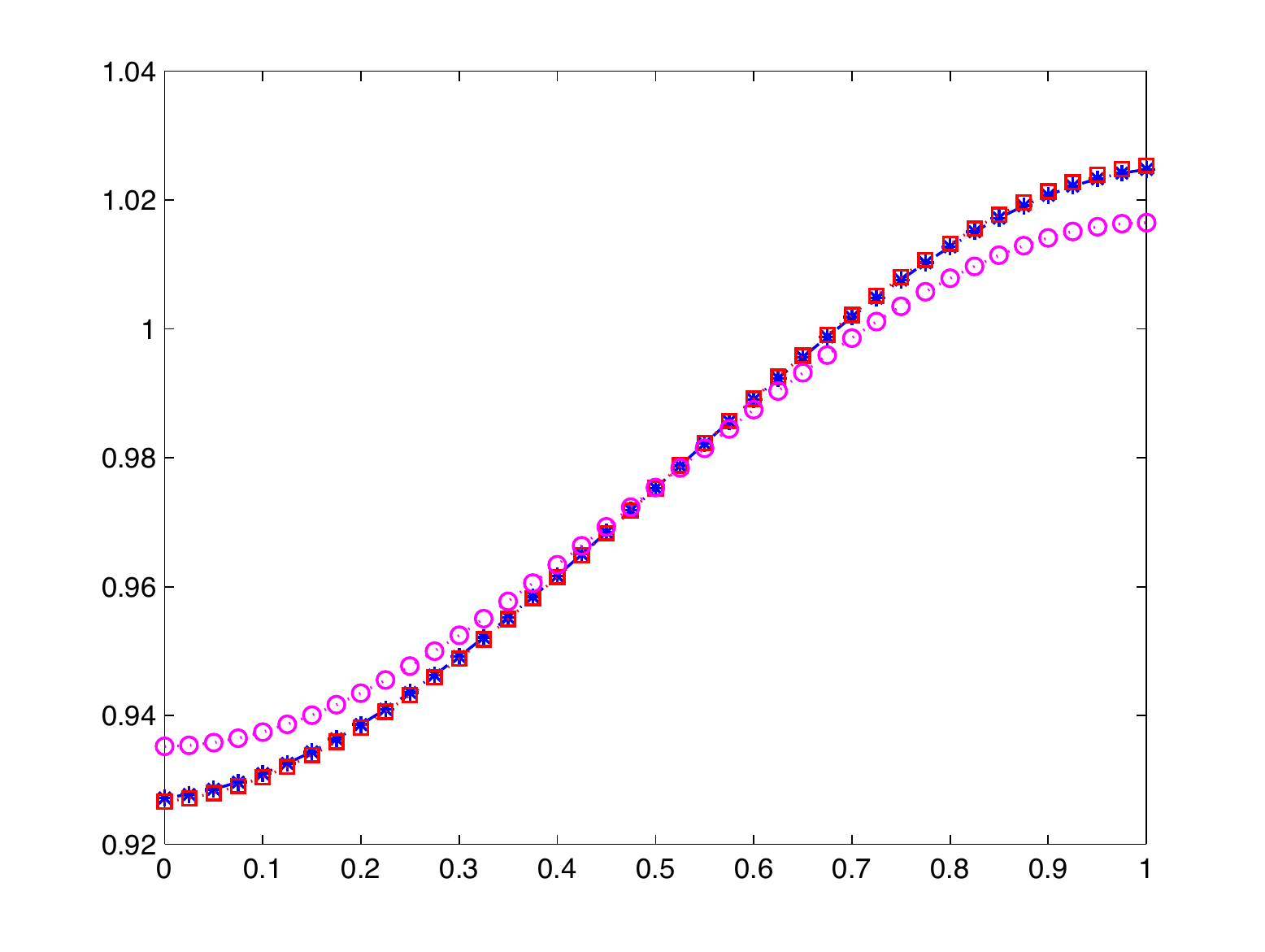}

\caption{This pulsed wave is obtained from the numerical resolution of  \fer{main2} with the fitness $a_2(x,\theta) = a_1(\theta) + \left( \sin(x) - \frac12 \right)$ and using the initial data and the parameters  given by \fer{iniex} and \fer{par}. The same conclusions as for the fitness $a_1$ hold.  Noticing that $H$ can take negative values in some zones which are unfavorable for the population, we observe that the population can pass through the obstacles and grow up in the favorable zones.}
\end{center}
\end{figure}

\begin{figure}[h]\label{Numconveps3}
\begin{center}
\includegraphics[width = 0.32\linewidth]{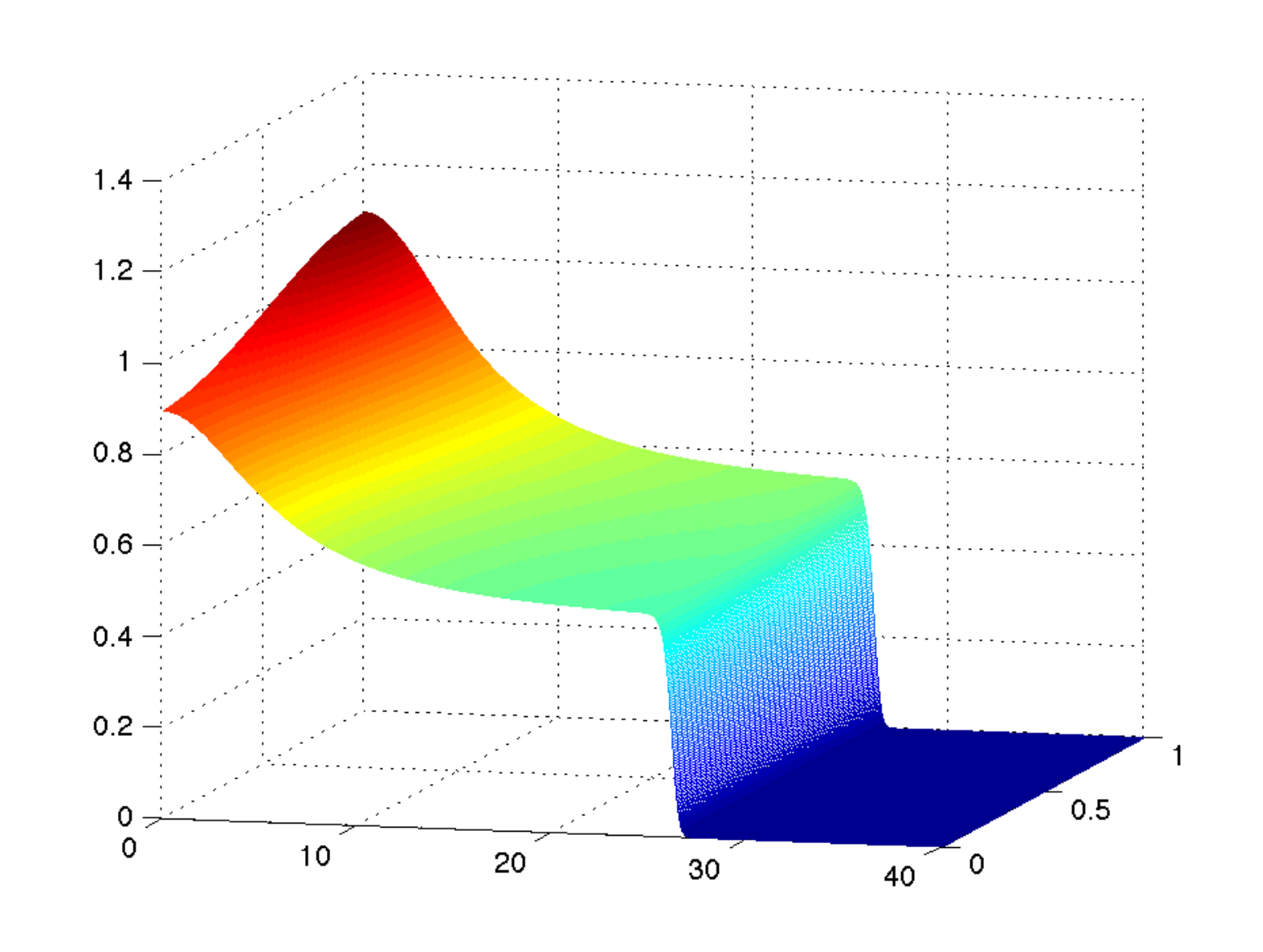}
\includegraphics[width = 0.32\linewidth]{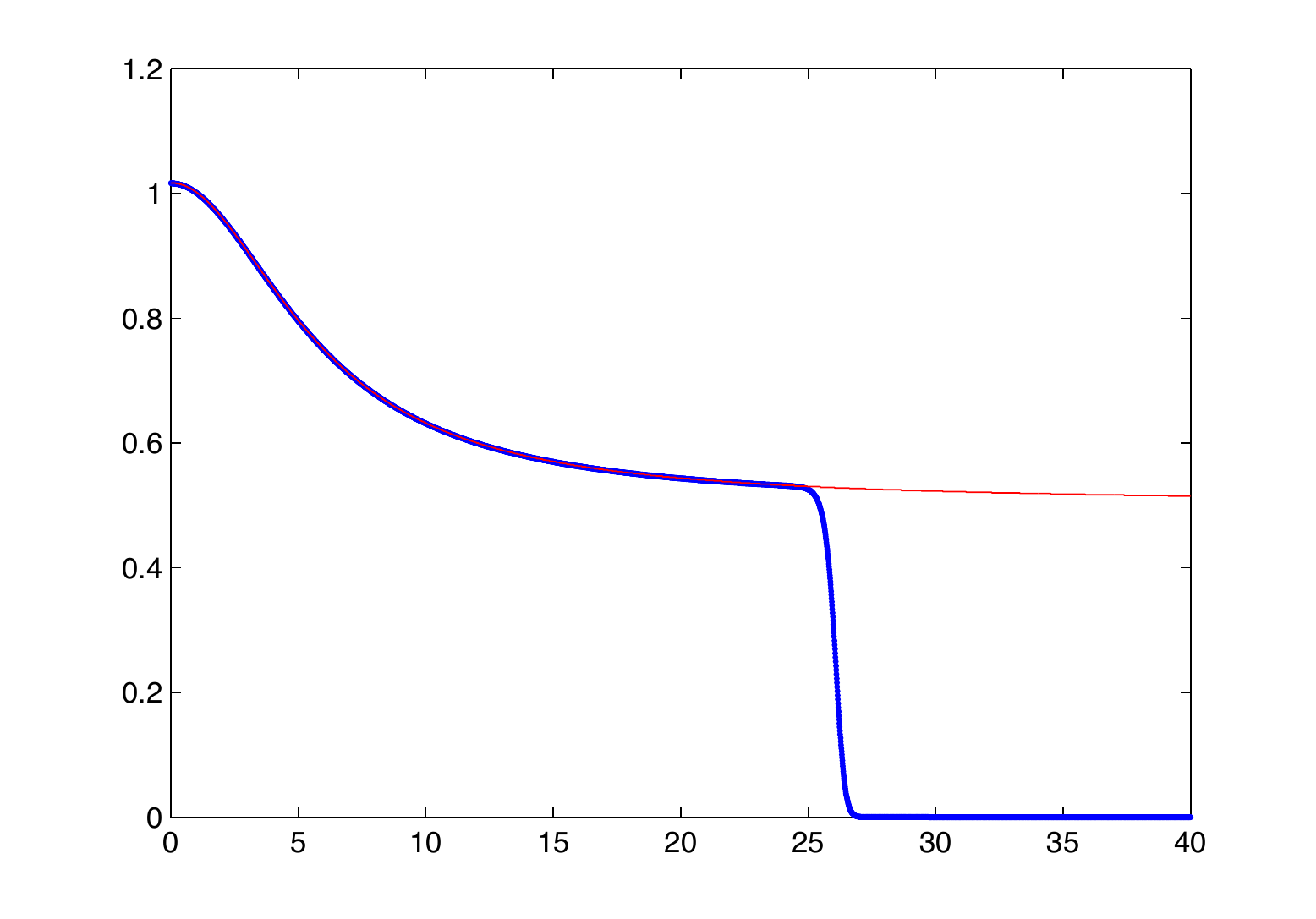}
\includegraphics[width = 0.32\linewidth]{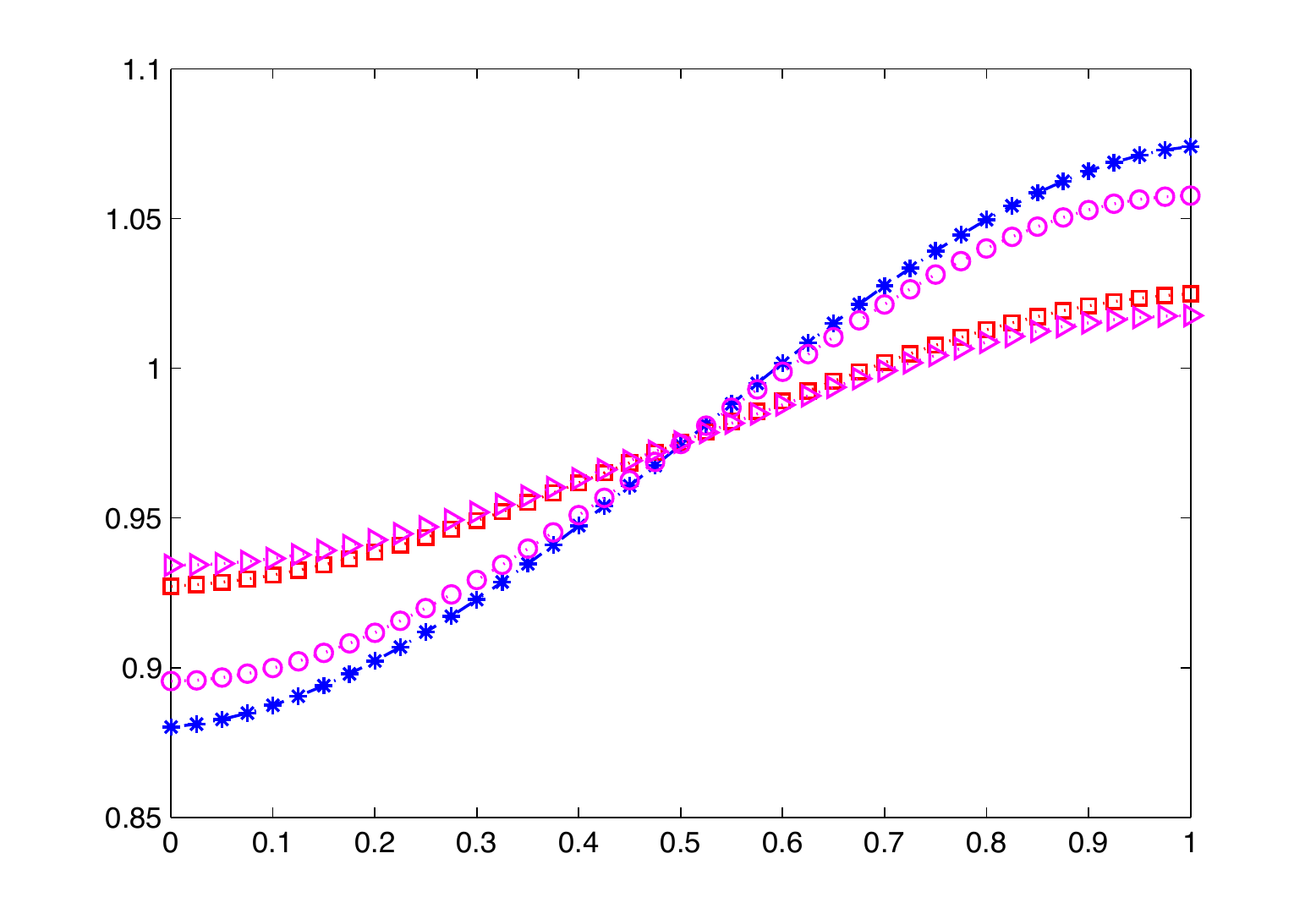}
\caption{We present numerical resolution of \eqref{main2} with the fitness $a_3(x,\theta) = a_1(\theta) \left( 1 + \frac{1}{1 + 0.05 x^2} \right)$ and using the initial data and the parameters  given by \fer{iniex} and \fer{par}. In this case, we have numerically obtained the Hamiltonian $H$, which depends nontrivially on the fitness $a_3$. We find again that the density $\rho_\eps$ converges towards $\frac{H(x)}{r}$. Finally, we also observe an error of  ${O}(\eps)$  between the renormalized trait distributions at the edge and the back of the front with the corresponding eigenfunctions $Q(x,\cdot)$.  }
\end{center}
\end{figure}

\section*{Acknowledgment}

S. M. wishes to thank Ga\"el Raoul for early discussions and computations on this problem.

\bibliography{bibli}
\bibliographystyle{plain}

%

\end{document}